\documentclass[11pt]{amsart}
\usepackage{hyperref}
\usepackage{tabularx,booktabs}
\usepackage{caption}
\usepackage{adjustbox}
\usepackage{blindtext}
\usepackage{etex}
\usepackage{amsmath}
\usepackage{amsfonts}
\usepackage{accents,cases}
\usepackage{algorithm} 
\usepackage{algorithmicx}
\usepackage{algpseudocode} 
\usepackage{graphicx} 
\usepackage{epstopdf}
\usepackage{amscd}
\usepackage{amsthm}
\usepackage{amssymb}
\usepackage{latexsym,float}
\usepackage{makecell}
\usepackage{multicol, multirow}
\usepackage{euscript}
\usepackage{epsfig}
\usepackage{subfigure}

\usepackage{graphics}
\usepackage{array}
\usepackage{enumerate}
\usepackage{color,tikz}
\usetikzlibrary{arrows,chains,matrix,positioning,scopes}
\makeatletter
\tikzset{join/.code=\tikzset{after node path={%
\ifx\tikzchainprevious\pgfutil@empty\else(\tikzchainprevious)%
edge[every join]#1(\tikzchaincurrent)\fi}}}
\makeatother
\tikzset{>=stealth',every on chain/.append style={join},
         every join/.style={->}}
\tikzstyle{labeled}=[execute at begin node=$\scriptstyle,
   execute at end node=$]
\usepackage{wasysym}
\usepackage{hyperref}
\usepackage{pdfsync}
\usepackage{comment}
\usepackage[all]{xy}

\allowdisplaybreaks[4]

\renewcommand{\vec}[1]{\mbox{\boldmath$#1$}}

\newcommand{\blue}{\textcolor{black}}

\newcommand{\bel}[1]{\begin{equation}\label{#1}}

\newcommand{\be}{\begin{equation}}
\newcommand{\sign}{\ensuremath{\mathrm{sign}}}

\newcommand{\sgn}{\ensuremath{\mathrm{Sgn}}}

\newcommand{\argmax}{\ensuremath{\mathrm{argmax\,}}}

\newcommand{\Hmm}[1]{\leavevmode{\marginpar{\tiny%
$\hbox to 0mm{\hspace*{-0.5mm}$\leftarrow$\hss}%
\vcenter{\vrule depth 0.1mm height 0.1mm width \the\marginparwidth}%
\hbox to
0mm{\hss$\rightarrow$\hspace*{-0.5mm}}$\\\relax\raggedright #1}}}

\newtheorem{theorem}{Theorem}[section]

\newtheorem{lemma}[theorem]{Lemma}
\newtheorem{cor}[theorem]{Corollary}

\newtheorem{remark}[theorem]{Remark}

\newcommand{\norm}[1]{\Vert #1\Vert_\infty}
\newcommand{\set}[1]{\{ #1\}}

\newcommand{\abs}[1]{\vert #1\vert}

\DeclareMathOperator*{\argmin}{\ensuremath{arg\,min}}
\DeclareMathOperator*{\cut}{\ensuremath{cut}}
\DeclareMathOperator*{\mean}{\ensuremath{mean}}
\begin{document}
 \title{A simple iterative algorithm for maxcut}

\author{Sihong Shao}
\email{sihong@math.pku.edu.cn}

\address{CAPT, LMAM and School of Mathematical Sciences,  Peking University, Beijing 100871, China}

\author{Dong Zhang}
\email{dongzhang@math.pku.edu.cn}

\address{LMAM and School of Mathematical Sciences,  Peking University, Beijing 100871, China}

\author{Weixi Zhang}
\email{zjwzrazwx@gmail.com}

\address{LMAM and School of Mathematical Sciences,  Peking University, Beijing 100871, China}

\begin{abstract}
We propose a simple iterative (SI) algorithm for the maxcut problem through fully using an equivalent continuous formulation. It does not need rounding at all and has advantages that all subproblems have explicit analytic solutions, the cut values are monotonically updated and the iteration points converge to a local optima in finite steps via an appropriate subgradient selection. Numerical experiments on G-set demonstrate the performance. In particular, the ratios between the best cut values achieved by SI and {those by some advanced combinatorial algorithms in [Ann.~Oper.~Res.~248 (2017) 365]} are at least  $0.986$ and can be further improved to at least $0.997$ by a preliminary attempt to break out of local optima.
\end{abstract}
 \maketitle
       \makeatletter
\@addtoreset{equation}{section}
\makeatother   
          
\tableofcontents




 \section{Introduction}
 \label{intro}

Given an undirected simple graph $G=(V,E)$ of order $n$ with the vertex
set $V$ and the edge set $E$,
a set pair $(S,S^\prime)$ is called a cut of $G$ if
$S\cap S^\prime = \varnothing$ and $S \cup S^\prime = V$.
The maxcut problem,
one of Karp's $21$ NP-complete problems \cite{Karp1972},
aims at finding a specific cut $(S,S^\prime)$ of $G$ to maximize
the cut value
\begin{equation}\label{eq:cut}
\cut(S)=\sum\limits_{\{i,j\}\in E(S,S^\prime)}w_{ij},
\end{equation}
where $E(S,S^\prime)$ collects all edges cross between $S$ and $S^\prime$ in $E$, 
and  $w_{ij}$ denotes the nonnegative weight on the edge $\{i,j\}\in E$.

Due to its widespread applications in various areas 
\cite{ChangDu1987,BarahonaGrotschelJungerReinelt1988,AizenbudShkolnisky2016},
several maxcut algorithms have been proposed to
search for approximate solutions and usually fall into two distinct categories:
discrete algorithms and continuous ones.
The former mainly refer to the combinatorial algorithms for maxcut,
which directly deal with the discrete objective function \eqref{eq:cut}
and usually adopt both complicated techniques to break out of local optima
and advanced heuristics 
for improving the solution quality,
such as the scatter search \cite{MartiDuarteLaguna2009}, the tabu search \cite{PalubeckisKrivickiene2004} and hybrid strategies within the framework of evolutionary algorithms \cite{LinGuan2016,MaHao2017}. In contrast, the objective functions for the latter,
often obtained from the relaxation of the discrete objective function \eqref{eq:cut},
are continuous, and thus standard continuous optimization algorithms can be applied into the relaxed problems in a straightforward manner, for instance, the Goemans-Williamson (GW) algorithm \cite{GoemansWilliamson1995},
{the CirCut algorithm resulted from the rank-two relaxation heuristics \cite{BurerMonteiroZhang2001},}
the spectral cut (SC) algorithm \cite{DelormePoljak1993,PoljakRendl1995}
and its recursive implementation (RSC) \cite{Trevisan2012,th:Ottaviano2008,ChangShaoZhang2016-maxcut}.
For all these continuous algorithms, rounding is essential and indispensable 
in obtaining a cut from a solution of the corresponding relaxed problem. 
The GW algorithm rounds the solution of a relaxing semidefinite programming via randomly selecting the hyperplanes until it achieves an expected cut value. 
{A deterministic strategy, named Procedure-CUT,  is adopted by CirCut to round the angle-vector solution to get a
best possible  associated cut.}
The SC algorithm obtains a cut by rounding the maximal eigenvector of graph Laplacian by a threshold, while the RSC algorithm recursively distributes part of unabsorbed points into two sets corresponding to a cut where the selection and assignment are determined by rounding the approximate solution 
of the dual Cheeger cut problem.





In this work, we propose a novel continuous algorithm for the maxcut problem, i.e., a simple iterative (SI) algorithm. Compared with the above-mentioned continuous maxcut algorithms, the proposed SI algorithm has the following distinct advantages. First, our inner subproblem can be solved analytically (see Theorem \ref{Thm:exact_solution}),
whereas no matter the GW algorithm or the RSC algorithm needs call other optimization solvers for the inner subproblems.
This constitutes the main reason why we use the adjunct word {\em simple} for the proposed algorithm.
Second, our continuous optimization problem is directly equivalent to the maxcut problem (see Theorem~\ref{thm:zero}), 
and the corresponding cut is updated in an iterative manner and converges to the local maximum (see Theorem~\ref{thm:conver_3}). That is, the SI algorithm does not need any rounding at all. 
Finally, as an iterative algorithm,  SI may select the cut by  SC  to be the initial point (see Section~\ref{sec:Numer-exper}). In other words, it can also be used to improve the quality of the solution obtained from any other algorithms. 


The rest is organized as follows. Section \ref{sec:equiv-conti} establishes 
an equivalent continuous formulation of the maxcut problem \eqref{eq:cut} and 
a Dinkelbach-type iterative algorithm with global convergence. 
However, the solvability of the related inner subproblem can not be assured due to both 
the NP-hardness and the lack of convexity. 
To this end, in Section \ref{sec:iter-scheme}, we propose our simple iterative algorithm with an analytical solution to the inner problem. Both cost analysis and quality check are performed through numerical experiments on G-set in Section \ref{sec:Numer-exper}.
Besides, in order to further improve the quality, a preliminary attempt to break out of local optima 
is also implemented there.
We are concluded with a few remarks in Section \ref{sec:conclusion}.

\bigskip

{\bf Acknowledgements.}
This work was supported by the National Key R~\&~D Program of China (Nos.~2020AAA0105200, 2022YFA1005102),  the National Natural Science Foundation of China (Nos.~12288101, 11822102) and China Postdoctoral Science Foundation (No.~BX201700009). SS is partially supported by Beijing Academy of Artificial Intelligence (BAAI). The authors would like to thank Professor Kung-Ching Chang 
for his long-term guidance, encouragement and support 
in mathematics, and useful comments on an earlier version of this paper. 
{The authors also thank Chuan Yang for her kind help in the comparison with CirCut as well as the anonymous referees for their valuable suggestions.}


\section{Equivalent continuous problems}
\label{sec:equiv-conti}


Given an undirected graph $G=(V,E)$ with nonnegative weights, 
let 
\begin{align}
\label{eq:I(x)}
I(\vec x)&=\sum_{\{i,j\}\in E} w_{ij}|x_i-x_j|,\\
\|\vec x\|_\infty & =\max\{|x_1|,\ldots,|x_n|\}, \label{eq:inf}\\
F(\vec x) &= \frac{I(\vec x)}{\norm{\vec x}}. \label{eq:F(x)}
\end{align}
It can be readily verified that 
the nonnegative function $F(\vec x)$ can achieve its maximum on $\mathbb{R}^n\setminus\{\vec 0\}$ provided by its homogeneity of degree zero. 

Let
\begin{align}
S^\pm(\vec x) & =\{i\in V:x_i=\pm \|\vec x\|_\infty\},  \label{eq:S+-}\\
S^<(\vec x) & =\{i\in V:|x_i|<\|\vec x\|_\infty\}. \label{eq:S0}
\end{align}
Then for any $\vec x^*\in \mathbb{R}^n \setminus \{\vec0\}$, we have 
\begin{equation}\label{eq:Mc}
M(\vec x^*):= \{\vec x\,\,\,|\,\,\,\|\vec x\|_\infty=\|\vec x^*\|_\infty,\;\;\; S^\pm(\vec x^*)\subset S^\pm(\vec x)\}
\end{equation}
is a convex polytope. In fact, the convexity of $I(\vec x)$ directly implies that, 
if $\vec x^*\in M(\vec x^*)$ is a maximizer of $F(\vec x)$ on $\mathbb{R}^n\setminus\{\vec 0\}$, so does any $\vec x\in M(\vec x^*)$.

For any nonempty subset $S\subset V$, 
we define an indicative vector $\vec 1_S$: 
\[
(\vec1_S)_i=\left\{\begin{array} {l}
1,\,\,\,\,\,\,\,\,\,\,\,i\in S,\\
0,\,\,\,\,\,\,\,\,\,\,\,i\notin S,
\end{array}\right.
\]
and then 
\[
\vec x = \vec 1_S -\vec 1_{V\setminus S}, 
\]
which satisfies
\begin{equation}\label{eq:FeqC0}
\frac{1}{2}F(\vec x) =\frac12 I(\vec x)=\sum_{\{i,j\}\in E}w_{ij} \frac{|x_i-x_j|}{2}=\cut(S).
\end{equation}


\begin{theorem}
The maxcut problem \eqref{eq:cut} can be rewritten into 
\label{thm:zero}
\begin{equation}\label{eq:maxcut-continuous}
\max_{S\subset V} \cut(S) = \frac12 \max\limits_{\vec x\in\mathbb{R}^n \setminus \{\vec0\}}F(\vec x),
\end{equation}
and any vector $\vec x^*$ reaching the maximum of $F(\vec x)$ produces a maxcut $(S,S^\prime)$ where 
the subset $S$ satisfies $ S^+(\vec x^*)\subset S\subset (S^-(\vec x^*))^c$. 
\end{theorem}

\begin{proof}
Combining $\cut(\varnothing)=0$ and Eq.~\eqref{eq:FeqC0} for the nonempty set situation  together leads directly to
\begin{equation}\label{eq:FgeC}
\frac12 \max\limits_{\vec x\in\mathbb{R}^n \setminus \{\vec0\}}F(\vec x)\ge\max_{S\subset V} \cut(S).
\end{equation}



On the other hand, suppose $\vec x^*$ is a maximizer of $F(\vec x)$ on $\mathbb{R}^n\setminus\{\vec 0\}$, i.e.,
\begin{equation}
\frac12 F({\vec x^*})= \frac12 \max\limits_{\vec x\in\mathbb{R}^n \setminus \{\vec0\}}F(\vec x).
\end{equation} 
For any $S^*$ satisfying $ S^+(\vec x^*)\subset S^*\subset (S^-(\vec x^*))^c$, there exists $\hat{\vec x}\in M(\vec x^*)$ defined in Eq.~\eqref{eq:Mc} such that
\begin{equation}
\frac{\hat{\vec x}}{\norm{\hat{\vec x}}} = \vec 1_{S^*}-\vec 1_{V\setminus S^*}
\end{equation}
also maximizes $F(\vec x)$ on $\mathbb{R}^n\setminus\{\vec 0\}$ thanks to the zeroth order homogeneousness of $F(\vec x)$.
That is, we have
\begin{equation}\label{eq:FleC}
\frac12 \max\limits_{\vec x\in\mathbb{R}^n \setminus \{\vec0\}}F(\vec x)=\frac12 F(\frac{\hat{\vec x}}{\norm{\hat{\vec x}}})=\cut(S^*) \le\max_{S\subset V} \cut(S),
\end{equation}
where Eq.~\eqref{eq:FeqC0} has been applied. 

The proof is finished as a result of Eqs.~\eqref{eq:FgeC} and \eqref{eq:FleC}.
\end{proof}

Theorem \ref{thm:zero} establishes an equivalent continuous optimization for the maxcut problem 
which will serve as the cornerstone of the subsequent design of an iterative algorithm.

Now we only need to consider 
\begin{equation}\label{eq:maxcut-continuous-1}
r_{\max} = \max\limits_{\vec x\in\mathbb{R}^n \setminus \{\vec0\}}F(\vec x).
\end{equation}
As shown in Theorem \ref{thm:conver_1}, it can be solved via the following so-called Dinkelbach iterative scheme \cite{Dinkelbach1967}, 
\begin{subequations}
\label{iter0}
\begin{numcases}
\vec x^{k+1}={\argmin\limits_{\|\vec x\|_p= 1} \{r^k \|\vec x\|_{\infty}-I(\vec x)\}}, \,\,\, p\in[1,\infty],\label{iter0-1}\\
r^{k+1}=F(\vec x^{k+1}),
\end{numcases}
\end{subequations}
where $\|\cdot\|_p$ denotes the standard $p$-norm in $\mathbb{R}^n$, i.e.,
\begin{equation}\label{eq:pnorm}
\|\vec x\|_p = (|x_1|^p+|x_2|^p+\cdots+|x_n|^p)^{\frac{1}{p}}.
\end{equation}

\begin{theorem}[global convergence]
\label{thm:conver_1}
The sequence $\{r^k\}$ generated by the iterative scheme \eqref{iter0} from any initial point $\vec x^0\in \mathbb{R}^n\setminus \{\vec0\}$ increases monotonically to the global maximum $r_{\max}$.
\end{theorem}

\begin{proof}
The definition of $\vec x^{k+1}$ (see Eq.~\eqref{iter0-1}) implies
\begin{equation*}
r^k \|\vec x\|_{\infty}-I(\vec x)\ge r^k\|\vec x^{k+1}\|_{\infty}-I(\vec x^{k+1}),\,\,\forall\, \vec x\,\text{ s.t. }\|\vec x\|_p= 1,
\end{equation*}
and substituting $\vec x=\vec x^k$ into the above inequality yields
\begin{equation*}
0=r^k \|\vec x^k\|_{\infty}-I(\vec x^k)\ge r^k\|\vec x^{k+1}\|_{\infty} -I(\vec x^{k+1}),
\end{equation*}
which means
\begin{equation*}
r^k\le r^{k+1}\le r_{\max}, \,\,\, \forall\, k\in \mathbb{N}^+.
\end{equation*}
Therefore
\[
\exists\, r^* \in [0,r_{\max}]\,\, \text { s.t. }\, \lim\limits_{k\to+\infty}r^k=r^*,
\]
and it suffices to show $r_{\max}\le r^*$. To this end, we denote
\begin{equation*}
f(r)=\min\limits_{\|\vec x\|_p= 1}(r\|\vec x\|_{\infty}-I(\vec x)),
\end{equation*}
which must be continuous on $\mathbb{R}$ by the compactness of the unit closed sphere $S_p=\{\vec x\in\mathbb{R}^n:\|\vec x\|_p= 1\}$. 

Note that
\begin{align*}
f(r^k)&=r^k\|\vec x^{k+1}\|_{\infty}-I(\vec x^{k+1})\\
&=r^k\|\vec x^{k+1}\|_{\infty}-r^{k+1}\|\vec x^{k+1}\|_{\infty}\\
&=\|\vec x^{k+1}\|_{\infty}
(r^k-r^{k+1})\to 0 \quad \text{as} \quad k\to +\infty,
\end{align*}
then we have
\begin{equation*}
f(r^*)=\lim\limits_{k\to+\infty}f(r^k)=0,
\end{equation*}
which implies
\begin{equation*}
r^* \|\vec x\|_{\infty} -I(\vec x)\ge 0, \,\,\forall\, \vec x\,\text{ with }\|\vec x\|_p= 1.
\end{equation*}
Hence, $\forall\, \vec x\in \mathbb{R}^n\setminus\{\vec0\}$, $$F(\vec x)=\frac{I(\vec x)}{\|\vec x\|_\infty}= \frac{I(\vec x)/\|\vec x\|_p}{\|\vec x\|_\infty/\|\vec x\|_p}= \frac{I(\hat{\vec x})}{\|\hat{\vec x}\|_\infty }\le r^*,$$
where $\hat{\vec x}=\vec x/\|\vec x\|_p$.  
\end{proof}

Although the inner subproblem \eqref{iter0-1} 
is not only non-convex but also
non-solvable in polynomial time due to the NP-hardness of the maxcut problem, 
the Dinkelbach scheme \eqref{iter0} provides us a good starting point to 
a feasible iterative algorithm for the maxcut problem.




\begin{remark}
Obviously, the equivalent continuous optimization \eqref{eq:maxcut-continuous-1} has a fractional form,
i.e., a ratio between two convex functions, but such kind of fractions have been hardly touched in the field of fractional programming \cite{SchaibleIbaraki1983}, where concave optimization problems, like optimizing the ratio of a concave function to a convex one, are usually considered. 
\end{remark}

\section{A simple iterative algorithm}
\label{sec:iter-scheme}



The non-convex subproblem \eqref{iter0-1} brings us a significant insight to deal with a relaxed subproblem alternatively, though it can not be solved in polynomial time. 

Denote the subgradient of $I(\vec x)$ by \cite{ChangShaoZhang2017}
\begin{equation}\label{eq:subg}
\partial I(\vec x) = \{\vec s=(s_1,\ldots,s_n) \big|\, s_i = \sum_{j:\{i,j\}\in E} w_{ij} z_{ij}, \,z_{ij} \in \sgn(x_i-x_j)\mbox{ and }z_{ij}=-z_{ji}\},
\end{equation}
where we have extended the sign function (note that $\sign(0)=1$ here)
\begin{equation}\label{signt}
\sign (t)=\begin{cases}
1, &\;\text{if }t\ge0,\\
-1,&\;\text{if }t<0,
\end{cases}
\end{equation}
into a set-valued function 
\begin{equation}\label{sgnt}
\sgn (t)=\begin{cases}
\{1\}, &\;\text{if }t>0,\\
\{-1\},&\;\text{if }t<0,\\
[-1,1],&\;\text{if }t=0.
\end{cases}
\end{equation}
For $\vec x=(x_1,x_2,\ldots,x_n)\in\mathbb{R}^n$, 
we are able to define corresponding vectorized versions
in an element-wise manner: 
\begin{align}
\sign(\vec x)&=(s_1, s_2,\ldots,s_n),\,\,\,s_i = \sign(x_i),\,i = 1,2,\ldots,n,\\
\sgn(\vec x)&=\{(s_1, s_2,\ldots,s_n) \big| s_i\in \sgn(x_i),\,i = 1,2,\ldots,n\}.
\end{align}

Since the function $I(\cdot)$ is convex, it holds
\begin{equation}
I (\vec x)\ge I(\vec y)+(\vec x-\vec y, \vec s)  , \,\,\, \forall\, \vec s\in\partial I (\vec y), \,\,\, \forall\, \vec x, \vec y \in \mathbb{R}^n,
\end{equation}
where $(\cdot,\cdot)$ denotes the standard inner product in $\mathbb{R}^n$. 
Further considering the fact that $I(\cdot)$ is homogeneous of degree one, we have 
\begin{equation}\label{eq:ieq}
I (\vec y)=(\vec y, \vec s)  , \,\,\, \forall\, \vec s\in\partial I (\vec y), \,\,\, \forall\, \vec y \in \mathbb{R}^n,
\end{equation}
and 
\begin{equation}\label{eq:rex}
I (\vec x)\geq  I(\vec y) +(\vec x-\vec y, \vec s) = (\vec x, \vec s), \,\,\, \forall\, \vec s\in\partial I (\vec y), \,\,\, \forall\, \vec x, \vec y \in \mathbb{R}^n.
\end{equation}

Plugging the relaxation \eqref{eq:rex} into the subproblem \eqref{iter0-1} 
modifies the two-step Dinkelbach iterative scheme into 
the following three-step one 
\begin{subequations}
\label{iter1}
\begin{numcases}{}
\vec x^{k+1}=\argmin\limits_{\|\vec x\|_p= 1} \{r^k \|\vec x\|_{\infty} - (\vec x,\vec s^k)\}, \,\,\, p\in[1,\infty], \label{eq:twostep_x2}
\\
r^{k+1}=F(\vec x^{k+1}),
\label{eq:twostep_r2}
\\
\vec s^{k+1}\in\partial I(\vec x^{k+1}),
\label{eq:twostep_s2}
\end{numcases}
\end{subequations}
with an initial data: 
$\vec x^0\in \mathbb{R}^n\setminus \{\vec0\}$, $r^0 = F(\vec x^0)$ and $\vec s^0\in\partial I(\vec x^0)$.

It can be readily verified that the subproblem \eqref{eq:twostep_x2} is convex via the relaxation \eqref{eq:rex} from Eq.~\eqref{iter0-1}. More importantly, we can write down a solution to the inner subproblem \eqref{eq:twostep_x2}
in an analytical manner (see Theorem \ref{Thm:exact_solution}). 
That is, no any other optimization solver is needed in \eqref{iter1} and so that a {\it simple} iterative algorithm we name it. Actually, we are able to prove that such iterative scheme \eqref{iter1} still keeps the monotonicity (see Theorem \ref{thm:conver_2}) and has local convergence (see Theorem \ref{thm:conver_3}).


%
%

\subsection{Exact solution to the inner subproblem}
\label{subsec:solution-inner}


Let
\begin{equation}\label{eq:L}
{L}(r,\vec v) := \min_{\|\vec x\|_p= 1}\{ r \norm{\vec x} -(\vec x,\vec v)\},
\,\,\,r\in\mathbb{R},\,\, \vec v\in\mathbb{R}^n, 
\end{equation}
denote the minimal value in Eq.~\eqref{eq:twostep_x2}.
In order to obtain the exact solution,
we need to show first a property of the simple iteration \eqref{iter1} (see Lemma \ref{lem:rs}),
and use it to narrow the scope of discussion.

\begin{lemma}
\label{lem:rs}
Suppose $\vec x^k$, $r^k$ and $\vec s^k$  are generated by the simple iteration \eqref{iter1}.
Then for $k\geq 1$, we always have $r^k\leq \Vert \vec s^k\Vert_1$. In particular, 
(1) $r^k= \Vert \vec s^k\Vert_1$ if and only if $\vec x^k/\norm{\vec x^k}\in \sgn(\vec s^k)$;
and (2) $r^k<\Vert \vec s^k\Vert_1$ if and only if $L(r^k,\vec s^k)<0$.
\end{lemma}

\begin{proof}
From H\"older's inequality,
\begin{equation}\label{eq:rls}
r^k\norm{\vec x}-(\vec x,\vec s^k) \ge r^k\norm{\vec x}-\Vert \vec s^k\Vert_1 \norm{\vec x} =(r^k-\Vert \vec s^k\Vert_1)\norm{\vec x},
\end{equation}
where the equality holds if and only if $\vec x\in \norm{\vec x}\sgn(\vec s^k)$. 

Plugging $\vec x= \vec x^k$ into Eq.~\eqref{eq:rls} and using Eq.~\eqref{eq:ieq} lead to
\begin{equation*}
(r^k-\Vert \vec s^k\Vert_1)\norm{\vec x^k}\le r^k\norm{\vec x^k}-(\vec x^k,\vec s^k) = r^k\norm{\vec x^k}- I(\vec x^k) = 0,
\end{equation*}
and thus $r^k\leq \Vert \vec s^k\Vert_1$.

Meanwhile, the statement (1) corresponds to 
the situation in which the equality holds. 

Next we will show the statement (2) is true.
On one hand, if $L(r^k,\vec s^k)<0$,
then there exists a vector $\vec x^{k+1}$ satisfying  
\begin{equation}\label{eq:dir1}
L(r^k,\vec s^k)=r^k \|\vec x^{k+1}\|_\infty-(\vec x^{k+1},\vec s^k)<0.
\end{equation}
Substituting $\vec x= \vec x^{k+1}$ into Eq.~\eqref{eq:rls} and using Eq.~\eqref{eq:dir1} yield
\[
(r^k-\Vert \vec s^k\Vert_1)\norm{\vec x^{k+1}}\le r^k\norm{\vec x^{k+1}}-(\vec x^{k+1},\vec s^k)=L(r^k,\vec s^k)<0, 
\]
and thus $r^k<\Vert \vec s^k\Vert_1$.

On the other hand, assume $r^k<\Vert \vec s^k\Vert_1$ holds. Choose $\vec y\in \sgn(\vec s^k)$ and let $\vec x^*  = {\vec y}/{\|\vec y\|_p}$.  It is obvious that $\norm{\vec y}=1$,  $\|\vec x^*\|_p =1$ and $(\vec y, \vec s^k) = \|\vec s^k\|_1$. In consequence, we have 
\begin{align*}
L(r^k,\vec s^k)&\le r^k \|\vec x^*\|_\infty-(\vec x^*,\vec s^k) \\  
&=r^k \|\frac{\vec y}{\|\vec y\|_p}\|_\infty-(\frac{\vec y}{\|\vec y\|_p},\vec s^k)\\
&=\frac{1}{\|\vec y\|_p}(r^k-\|\vec s^k\|_1)<0.
\end{align*}

The proof is completed. 
\end{proof}

%
%

Given a real number $r>0$ and a vector $\vec v=(v_1,\cdots,v_n)\in\mathbb{R}^n$,
in view of Lemma \ref{lem:rs} and the subproblem \eqref{eq:twostep_x2}, 
we only need to consider the minimization problem
\begin{equation}\label{op}
\vec x^*=\argmin_{\|\vec x\|_p= 1}\{ r \norm{\vec x} -(\vec x,\vec v)\},
\end{equation}
under the condition
\begin{equation}\label{eq:con}
0<r\le\|\vec v\|_1.
\end{equation}

Without loss of generality, it suffices to discuss an ordered situation:
\begin{equation}\label{eq:order0}
\abs{v_1}\ge \abs{v_2}\ge\cdots  \ge \abs{v_n}\ge \abs{v_{n+1}}=0,
\end{equation}
where we have extended the index set into $\{1,\ldots,n+1\}$ and introduced an auxiliary element $v_{n+1}=0$ for convenience.
Consider the accumulation of increment 
\begin{equation}\label{eq:A}
A(m) = \sum_{j=1}^m (|v_j|-|v_{m+1}|), \,\,\, m\in\{1,2,\ldots,n\},
\end{equation}
and then from Eq.~\eqref{eq:order0} we have 
\begin{equation}
0=A(0)\leq A(1)\leq A(2) \leq \cdots \leq A(n) = \|\vec v\|_1. 
\end{equation}
Meanwhile, we need a key step to increase the regularity for $p\in(1,\infty)$ through rewriting 
the minimization problem \eqref{op} into
\begin{align}
\min_{\|\vec x\|_p= 1}\{ r \norm{\vec x} -(\vec x,\vec v)\} &= 
\min_{\|\vec u\|_p= 1}\{ r \norm{\vec u} -(\vec u,|\vec v|)\}  \label{eq:min0}\\
&= \min_{\vec w \neq \vec 0}\frac{ r \norm{\vec w} -(\vec w,|\vec v|)}{\|\vec w\|_p} \\
&= \min_{\norm{\vec z}= 1} \frac{ r  -(\vec z,|\vec v|)}{\|\vec z\|_p}\\
&= \min_{{\vec z}\in B_\infty} G(\vec z),
\label{eq:minG}
\end{align}
where 
\begin{align}
G(\vec z) &= \frac{r  -(\vec z,|\vec v|)}{\|\vec z\|_p}, \\
B_\infty &= \{\vec z \in \mathbb{R}^n: \|\vec z\|_\infty \le 1\},
\end{align}
and the absolute value is taken in an element-wise manner, e.g.,  
$|\vec v|  =(|v_1|,|v_2|,\ldots, |v_n|)$. 
Here we have used extensively the fact that 
$\|\cdot\|_p$,  $\|\cdot\|_\infty$ and $(\cdot,|\vec v|)$
are all homogeneous of degree one,
and in Eq. \eqref{eq:minG} the fact that 
$G(\vec z)$ achieves its minimum on the boundary of the feasible region $B_\infty$
due to 
\begin{equation}
G(\lambda\vec z)-G(\vec z) =(\frac{1}{\lambda}-1)\frac{r}{\|\vec z\|_p}<0 \,\,\,\, \text{for}\,\,\,\, \lambda>1.
\end{equation}
If one denotes the corresponding minimizers in Eqs.~\eqref{eq:min0}-\eqref{eq:minG} by $\vec x^*$,  
$\vec u^*$, $\vec w^*$, and $\vec z^*$, respectively, 
then we have
\begin{align}
\vec u^* &= \text{sign}(\vec v) \cdot \vec x^*, \label{eq:dot} \\
\vec w^* &= \|\vec w^*\|_p \vec u^*, \label{eq:w}\\
\vec z^* &= \frac{\vec w^*}{\norm{\vec w^*}}, \label{eq:z}
\end{align}
where $\vec x\cdot \vec y$ in Eq.~\eqref{eq:dot} denotes element-by-element multiplication of vectors $\vec x$ and $\vec y$.  
It is obvious that we only need to search for the minimizer $\vec z^*$
to the minimization problem \eqref{eq:minG},
with which we are able to reach our target 
\begin{equation}\label{eq:case1-sol}
\vec x^* = \frac{\text{sign}(\vec v) \cdot \vec z^*}{\|\vec z^*\|_p},
\end{equation}
by virtue of   
the one-to-one mappings \eqref{eq:dot}, \eqref{eq:w} and \eqref{eq:z}.


According to the condition \eqref{eq:con},  
the rest of the discussion falls into the following three scenarios. 
Before that, we need the following lemma to characterize the minimizer of $G(\vec z)$ on $B_\infty$,
denoted by $\vec z^*=(z_1^*,z_2^*,\ldots,z_n^*)$.

\begin{lemma}
\label{lem:z*}
Let $\vec z^*=(z_1^*,z_2^*,\ldots,z_n^*)$ be the minimizer of $G(\vec z)$ on $B_\infty$. 
If $r<\|\vec v\|_1$ and $1\leq p<\infty$, then $\vec z^*$ has not only nonnegative elements, 
but also the same order as $|\vec v|$.
\end{lemma}

\begin{proof}
Since $\vec z^*=(z_1^*,z_2^*,\ldots,z_n^*)$ is the minimizer of $G(\vec z)$ on $B_\infty$,  it holds
$$
G(\vec z^*) \le G(\vec z),\;\;\;\forall\,\,\vec z\in B_\infty.
$$ 
In the following, the proof is by contradiction and split into three steps. 

First, 
the condition $r<\|\vec v\|_1$ directly implies that $\vec z^*$ satisfies 
\begin{equation}\label{eq:neg0}
r  -(\vec z^*,|\vec v|) = G(\vec z^*) \|\vec z^*\|_p \le G(\vec 1) \|\vec z^*\|_p=\frac{\|\vec z^*\|_p}{\|\vec 1\|_p}(r -\|\vec v\|_1) < 0,
\end{equation}
where $\vec 1\in B_\infty$.

Second, if there exists some $i\in \{1,\cdots,n\}$ such that $z_i^*<0$,
and let $\hat{\vec z} = \{\hat{z}_1,\cdots,\hat{z}_n\}$ be a vector satisfying: $\hat{z}_i=0$ and $\hat{z}_j = z_j^*$ for $j\ne i$, 
then we have $\hat{\vec z}\in B_\infty$ and
\begin{equation}\label{eq:znonneg}
G(\hat{\vec z}) =  \frac{r  -(\hat{\vec z},|\vec v|)}{\|\hat{\vec z}\|_p} \le \frac{r  -(\vec z^*,|\vec v|)}{\|\hat{\vec z}\|_p}<\frac{r  -(\vec z^*,|\vec v|)}{\|\vec z^*\|_p} = G(\vec z^*),
\end{equation}
from \eqref{eq:neg0}. 
This contradicts with the fact that $\vec z^*$ is the minimizer of $G(\vec z)$ on $B_\infty$. 


Third, if there exists $i,j  \in\{1,\cdots,n\}$ such that $(z_i^*-z_j^*)(|v_i|-|v_j|)<0$, Let $\hat{\vec z}$ be a vector obtained by exchanging the $i$-th and the $j$-th elements of $\vec z$, then $\hat{\vec z}\in B_\infty$, $\|\hat{\vec z}\|_p = \|\vec z^*\|_p$, and 
\begin{equation}
G(\hat{\vec z}) -G(\vec z^*)  =\frac{(z_i^*-z_j^*)(|v_i|-|v_j|)} {\|\vec z\|_p}<0.
\end{equation}
This also contradicts with the fact that  $\vec z^*$ achieves the minimum of $G(\vec z)$ on $B_\infty$.
\end{proof}

More importantly, 
under the conditions of Lemma~\ref{lem:z*}, 
the objective function $G(\vec z)$ is now differentiable 
with the $i$-th partial derivative being
\begin{equation}\label{eq:pg2}
\frac{\partial G(\vec z^*)}{\partial z_i} = - \frac{|v_i|\|\vec z^*\|_p+(r  -(\vec z^*,|\vec v|))\|\vec z^*\|_p^{1-p}(z_i^*)^{p-1}}{\|\vec z^*\|_p^2}. 
\end{equation}

$\bullet$ {\bf Scenario 1}: 
\begin{equation}\label{case1}
r<\|\vec v\|_1,\text{ and } \;\;\;1< p<\infty.
\end{equation}

In view of Lemma \ref{lem:z*} and the decreasing order of $|\vec v|$ described in Eq.~\eqref{eq:order0}, we may assume that there exists an integer $m_0\in\{1,2,\ldots,n\}$ such that $\vec z^* $ satisfies
\begin{equation}\label{eq:order}
1=z_1^*=z_2^*=\cdots=z^*_{m_0}>z_{m_0+1}^*\ge \cdots \ge z^*_n\ge 0.
\end{equation}

Let 
\begin{equation}
{T}= \{\vec t\in\mathbb{R}^n:  \,\vec z^*+\epsilon \vec t \in B_\infty \text{ for sufficiently small } \epsilon >0\}
\end{equation}
denote the tangent cone of $B_\infty$ at $\vec z^*$.

From Eq.~\eqref{eq:order},  the tangent cone can be readily rewritten into
\begin{equation}\label{eq:T2}
T= \{\vec t=(t_1,t_2,\ldots,t_n)\in\mathbb{R}^n:\, t_i\le 0,\,\,i= 1, 2, \ldots,m_0\}.
\end{equation}

Since the minimizer $\vec z^*$ achieves a local minimum of $G(\vec z)$ on $B_\infty$, 
according to Eq.~\eqref{eq:T2}, 
we have 
\begin{align}
& \frac{\partial G(\vec z^*)}{\partial \vec t}\ge 0, \,\,\, \forall\, \vec t\in T  \\
\Leftrightarrow \quad &
\begin{cases}
\displaystyle \frac{\partial G(\vec z^*)}{\partial z_i} \le 0,\,\,\,\, i = 1, 2, \ldots,m_0, \\
\displaystyle \frac{\partial G(\vec z^*)}{\partial z_i} = 0, \,\,\,\, i = m_0+1,\ldots, n.
\end{cases} \label{eq:neqpg}
\end{align}

For the sake of subsequent discussion, 
it is convenient to introduce the following three auxiliary quantities:
\begin{align}
\alpha &= \sum_{i=1}^{m_0}|v_i|-r, \\
\beta &= \sum_{i=m_0+1}^nz_i^*|v_i|, \label{beta}\\
\gamma &= \sum_{i=m_0+1}^n (z_i^*)^p, \label{gamma}
\end{align}
and then 
\begin{align}
\alpha+\beta &= (\vec z^*,|\vec v|)-r>0,\\
m_0+\gamma &= \|\vec z^*\|_p^p>0.
\end{align}
Therefore, the partial derivative \eqref{eq:pg2} becomes
\begin{equation}\label{eq:pg3}
\frac{\partial G(\vec z^*)}{\partial z_i} = \frac{1}{\|\vec z^*\|_p} \left({\frac{\alpha+\beta}{m_0+\gamma}(z_i^*)^{p-1}-|v_i|}\right),
\end{equation}
and then Eq.~\eqref{eq:neqpg} directly implies 
\begin{equation} 
\label{eq:zpm1}
\begin{cases}
\displaystyle (z_i^*)^{p-1} \le \frac{m_0+\gamma}{\alpha+\beta}|v_i|,\,\,\,\, i=1,2\ldots,m_0,\\
\displaystyle (z_i^*)^{p-1} = \frac{m_0+\gamma}{\alpha+\beta}|v_i| ,\,\,\,\, i=m_0+1,\ldots,n. 
\end{cases}
\end{equation}

Substituting Eq.~\eqref{eq:zpm1} into Eq.~\eqref{gamma} and using  Eq.~\eqref{beta} yields
\begin{align}
\gamma &= \sum_{i=m_0+1}^n (z_i^*)^{p-1}z_i^* \\
 &= \sum_{i=m_0+1}^n \frac{m_0+\gamma}{\alpha+\beta}|v_i|z_i^* \\
 &=\frac{m_0+\gamma}{\alpha+\beta}\beta,
\end{align}
and then from $m_0>0$, we have
\begin{equation} 
\frac{m_0+\gamma}{\alpha+\beta}=\frac{m_0}{\alpha},\;\;\;\text{and}\;\;\;\alpha>0.
\end{equation}
That is, the condition for local minimizers, Eq.~\eqref{eq:neqpg} or Eq.~\eqref{eq:zpm1}, 
can be further simplified into 
\begin{equation}\label{eq:zpm2}
\begin{cases}
\displaystyle (z_i^*)^{p-1} \le a_i,\,\,\,\,i=1,2\ldots,m_0,\\
\displaystyle (z_i^*)^{p-1} = a_i,\,\,\,\,i=m_0+1,\ldots,n,
\end{cases}
\end{equation} 
where 
\begin{equation}\label{eq:a}
a_i =\frac{m_0}{\alpha}|v_i|=\frac{m_0|v_i|}{\sum_{j=1}^{m_0}\abs{v_j}-r}, \,\,\,\, i=1, 2, \ldots, n.
\end{equation}

Combining Eqs.~\eqref{eq:order} and \eqref{eq:zpm2} determines the minimizer $\vec z^* = (z_1^*, z_2^*,\ldots, z_n^*)$:  
\begin{equation}\label{sol1}
z_i^* = \min\{1,{a_i}^{\frac{1}{p-1}}\}, \,\,\,\, i=1, 2, \ldots, n, 
\end{equation}
and then we are able to reach our target $\vec x^*$ by Eq.~\eqref{eq:case1-sol}.

Therefore the exact solution to the inner subproblem for \textbf{Scenario 1} 
has been obtained. The only remaining thing is how to determine $m_0$ efficiently used in Eq.~\eqref{eq:a}.
This can be achieved with the help of the accumulation function $A(m)$ defined in Eq.~\eqref{eq:A}.
Namely, $m_0$ is the smallest integer $m$ satisfying $A(m)>r$
\begin{equation}\label{eq:defm0}
m_0= \min\{m\in\{1,2,\ldots,n\}: A(m)> r\},
\end{equation}
which can be readily verified as follows,
\begin{align*}
& a_{m_0}\ge  (z^*_{m_0})^{p-1}=1>(z^*_{m_0+1})^{p-1}=a_{m_0+1} \\
\Leftrightarrow \quad &\frac{m_0|v_{m_0}|}{\sum_{j=1}^{m_0}\abs{v_j}-r}\ge 1>\frac{m_0|v_{m_0+1}|}{\sum_{j=1}^{m_0}\abs{v_j}-r}\\
\Leftrightarrow\quad &\sum_{j=1}^{m_0} (\abs{v_j}-\abs{v_{m_0+1}})> r\ge \sum_{j=1}^{m_0-1} (\abs{v_j}-\abs{v_{m_0}})\\
\Leftrightarrow \quad &   A(m_0)> r\ge A(m_0-1).
\end{align*}

$\bullet$ {\bf Scenario 2}: 
\begin{equation}\label{case2}
r<\|\vec v\|_1,\text{ and } \;\;\;p = 1.
\end{equation}

In this scenario, we still have the minimizer $\vec z^*=(z_1^*,z_2^*,\ldots,z_n^*)$ has nonnegative elements according to Lemma \ref{lem:z*},
namely, $\forall\,i\in\{1,2,\ldots,n\}, 0\le z_i^*\le1$. 

Let $m_1$ be the largest integer satisfying $A(m-1)<r$, i.e.,
\begin{equation}\label{eq:defm1}
m_1 = \max\{m\in\{1,2,\ldots,n\}: A(m-1)<r\},
\end{equation}
where $A(m)$ is the accumulation function defined by Eq.~\eqref{eq:A}.
Comparing Eq. \eqref{eq:defm1} with Eq.~\eqref{eq:defm0}, 
we have $m_1\leq m_0$ where $m_0$ is defined by Eq.~\eqref{eq:defm0}, and specifically 
\begin{equation}
A(m_1-1)<r=A(m_1 ) =A(m_1+1) = \cdots =A(m_0-1) <A(m_0),
\end{equation}
thereby indicating 
\begin{equation}\label{eq:order2}
|v_{m_1}| > 
|v_{m_1+1}| = \cdots =|v_{m_0}|>|v_{m_0+1}|.
\end{equation}

For $i\in\{1,2,\ldots,n\}$, consider a continuous function on $[0,1]$: 
\begin{equation}
G_i(t) = G(z_1^*,\ldots, z_{i-1}^*, t,z_{i+1}^*,\ldots z_n^*)=-|v_i|+\frac{r-\sum_{j\ne i}z_j^*(|v_j|-|v_i|)}{\sum_{j\ne i}|z_j^*|+t}.
\end{equation}
Since $\vec z^*$ is a minimizer of $G(\vec z)$ on $B_\infty$, 
it can be readily verified that $\min_{t\in[0,1]} G_i(t) = G(\vec z^*)$, i.e.,
\begin{equation}\label{eq:git}
z_i^* \in \argmin_{t\in[0,1]} G_i(t),\;\;\;\;\forall\, i\in\{1,2,\ldots,n\},
\end{equation}
from which we are able to determine $z_i^*$.
The related discussion needs to split the index set into the following three cases.


\begin{enumerate}[(1)]
\item For $1\le i\le m_1$, we claim that $G_i(t)$ is a strictly monotonically decreasing function due to 
$r-\sum_{j\ne i}z_j^*(|v_j|-|v_i|)>0$ and thus $z_i^*=1$. 
The verification shown below is in a straightforward manner: 
\begin{align*}
r-\sum_{j\ne i}z_j^*(|v_j|-|v_i|)&\ge r-\sum_{j=1}^{i-1}z_j^*(|v_j|-|v_i|)\\
& \ge r-\sum_{j=1}^{i-1}(|v_j|-|v_i|) = r-A(i-1)>0.
\end{align*}
 
\item For $m_0< i\le n$, we claim that $z_i^*=0$. Suppose the contrary that $z^*_{i_0}>0$ is the positive element with the largest index and $i_0>m_0$. Then  the order given in Eqs.~\eqref{eq:order0} and \eqref{eq:order2} directly implies 
\begin{align*}
r-\sum_{j\ne i_0}z_j^*(|v_j|-|v_{i_0}|)&= r-\sum_{j=1}^{m_1}(|v_j|-|v_{i_0}|)-\sum_{j=m_1+1}^{i_0-1}z_j^*(|v_j|-|v_{i_0}|)\\
& \le r-\sum_{j=1}^{m_1}(|v_j|-|v_{i_0}|)\\
 &
 \begin{cases}
 \displaystyle < r-\sum_{j=1}^{m_1}(|v_j|-|v_{m_1+1}|) =r-A(m_1)= 0, \,\,\, \text{if } m_1<m_0, \\
\displaystyle  \leq r-\sum_{j=1}^{m_1}(|v_j|-|v_{m_1+1}|) =r-A(m_0)< 0, \,\,\, \text{if } m_1 = m_0, 
  \end{cases}
\end{align*}
and thus we have $G_{i_0}(t)$ is a strictly monotonically increasing function. 
That is, the minimizer of $G_{i_0}(t)$ on $[0,1]$ is $0$, i.e., $z_{i_0}^*=0$,
which is obviously a contradiction. 

\item For $m_1<i\le m_0$, using the results for above two cases and  
the order \eqref{eq:order2} yields 
\begin{align*}
r-\sum_{j\ne i}z_j^*(|v_j|-|v_i|) &= r-\sum_{j=1}^{m_1}(|v_j|-|v_i|)-\sum_{j=m_1+1}^{m_0} z_j^* (|v_j|-|v_i|)\\
&=r-\sum_{j=1}^{m_1}(|v_j|-|v_{m_1+1}|)=r-A(m_1)= 0,
\end{align*}
and thus we have $G_i(t)$ is a constant function on $[0,1]$, i.e., $G_i(t) \equiv -|v_i|$.
That is, the minimizer $z_i^*$ can take any value in $[0,1]$. 
%
\end{enumerate}

In a word, the minimizers $\vec z^*=(z_1^*,\ldots,z_n^*)$ for \textbf{Scenario 2} constitute the following set
\begin{equation}\label{sol2}
\{(z_1^*,\ldots,z_n^*)\in[0,1]^n \big | \,\,\,z^*_i=1,i=1,\ldots,m_1, \,\,\,\text{and }\,\,\,z^*_i=0, i=m_0+1,\ldots,n\}.
\end{equation}
and thus $\vec x^*$ can be also obtained through Eq.~\eqref{eq:case1-sol}.

$\bullet$ {\bf Scenario 3}: 
\begin{equation}\label{case3}
r =\|\vec v\|_1, \;\;\; \text{ or }  \;\;\; p = \infty.
\end{equation}

For $r =\|\vec v\|_1$, according to Lemma \ref{lem:rs}, we have that $\vec x^*$ is a minimizer of problem \eqref{op} if and only if
\begin{equation}\label{eq:xsc3}
\vec x^* / \|\vec x^*\|_\infty\in\sgn(\vec v),\,\,\,\, \|\vec x^*\|_p = 1.
\end{equation}

For $p=\infty$, the minimization problem Eq.~\eqref{op} becomes a linear optimization problem:  
\begin{equation}\label{eq:opinfty}
\vec x^* =\argmin_{\|\vec x\|_\infty= 1}\{ r -(\vec x,\vec v)\},
\end{equation}
the solution of which can be still represented by Eq.~\eqref{eq:xsc3}. 

Hence the solution Eq.~\eqref{eq:xsc3} solves the minimization problem \eqref{op}
for {\bf Scenario 3}.


Summarizing the above analysis for three scenarios, we have figured out the exact solution to the inner subproblem \eqref{eq:twostep_x2},  as stated in the following theorem.  

\begin{theorem}[exact solution]
\label{Thm:exact_solution}
The solution of the minimization problem \eqref{op} under the condition \eqref{eq:con}
can be expressed analytically in Eqs.~\eqref{eq:case1-sol} and \eqref{sol1} for $r<\|\vec v\|_1$, and $ 1< p<\infty$,  
Eqs.~\eqref{eq:case1-sol} and \eqref{sol2} for $r<\|\vec v\|_1$, and $ p=1$, and  Eq.~\eqref{eq:xsc3} otherwise. 
\end{theorem}

Hereafter we use a set denoted by $X^{k+1}_p$ to collect all those analytical solutions
of the inner subproblem \eqref{eq:twostep_x2}, i.e., $\vec x^{k+1}\in X_p^{k+1}$.
Obviously, $X_p^{k+1}$ is closed.

Using the above analytical solution,
we are able to get a lower bound for $F(\vec x)$ below, 
which will be useful in the subsequent convergence analysis. 

\begin{cor}\label{cor:nondec}
The minimizer $\vec x^*$ to the problem \eqref{op} under the condition \eqref{eq:con} 
satisfies 
 \begin{equation}\label{eq:m}
 \frac{(\vec x^*,\vec v)}{\norm{\vec x^*}} \ge \sum_{i=1}^{m} |v_i|,
 \end{equation}
 where $m =m_0$ for both {\bf Scenario 1} and {\bf Scenario 2}, and $m=n$ for {\bf Scenario 3}.
\end{cor}

\begin{proof}
For both {\bf Scenario 1} and {\bf Scenario 2}, there exists a minimizer $\vec z^* =( z^*_1, z_2^*, \ldots,z^*_n)\in B_\infty \cap \mathbb{R}^n_+$ to the equivalent problem \eqref{eq:minG},
which satisfies  
$$
z^*_1=\cdots  = z^*_{m_0}=1,
$$
according to Eqs.~\eqref{sol1} and \eqref{sol2}, respectively. 
Then, from Eq.~\eqref{eq:case1-sol}, we have 
\begin{equation}
 \frac{(\vec x^*,\vec v)}{\norm{\vec x^*}} =(\sign(\vec v)\cdot\vec z^*,\vec v) =(\vec z^*,|\vec v|)\ge \sum_{i=1}^{m_0} |v_i|.
\end{equation}
As for {\bf Scenario 3}, using Eq.~\eqref{eq:xsc3}, it can be easily verified that
\begin{equation}
 \frac{(\vec x^*,\vec v)}{\norm{\vec x^*}} =  (\frac{\vec x^*}{\norm{\vec x^*}},\vec v)=(\sgn(\vec v),\vec v) = \|\vec v\|_1 =\sum_{i=1}^{n} |v_i|.
\end{equation}
Thus the proof is finished. 
\end{proof}

\subsection{Subgradient selection and convergence analysis}


Besides the inner subproblem \eqref{eq:twostep_x2},
another key issue to implement the simple iteration \eqref{iter1}
is how to choose the subgradient (see Eq.~\eqref{eq:twostep_s2}).
For a general selection, 
according to Eqs.~\eqref{eq:ieq},  \eqref{eq:rex} and \eqref{iter1},
we are able to obtain $r^k \le r^{k+1}$, 
because 
\begin{align*}
0 = r^k\|\vec x^k\|_\infty -I(\vec x^k) = &r^k\|\vec x^k\|_\infty-( \vec x^k,\vec s^k) \\
\ge & r^k\|\vec x^{k+1}\|_\infty-(\vec x^{k+1}, \vec s^k) \\
\ge & r^k\|\vec x^{k+1}\|_\infty- I(\vec x^{k+1}) \\
=    & \|\vec x^{k+1}\|_\infty(r^k-r^{k+1}).
\end{align*}

\begin{theorem}[monotonicity]
\label{thm:conver_2}
The sequence $\{r^k\}$ generated by the iterative scheme \eqref{iter1} from any initial point $\vec x^0\in \mathbb{R}^n\setminus \{\vec0\}$ increases monotonically.
\end{theorem}




The monotone increasing in Theorem \ref{thm:conver_2} is also a direct consequence of Lemma \ref{lem:rs} and Corollary \ref{cor:nondec} because there exists a index set $\iota\subset\{1,2,\ldots,n\}$, the size of which should be $m$ in Eq.~\eqref{eq:m} (the upper bound of summation index), such that 
\begin{equation}\label{eq:rsr}
r^{k+1}=\frac{I(\vec x^{k+1})}{\norm{\vec x^{k+1}}}\ge \frac{(\vec x^{k+1},\vec s^k)}{\norm{\vec x^{k+1}}}\ge \sum_{i\in \iota} |s_i^k|\ge r^k,
\end{equation}
where the subgradient $\vec s^k$ is not required to be ordered like Eq.~\eqref{eq:order0}.
In particular, such monotone increasing could be strict, i.e., $r^{k+1} > r^{k}$ via improving the last `$\ge$' in Eq.~\eqref{eq:rsr} into `$>$',  if $\Vert \vec s^k\Vert_1>r^k$  holds in each step.
To this end, we only need to determine a subgradient $\vec s^\sigma\in \partial I(\vec x^k)$ such that $\|\vec s^\sigma\|_1>r^k$ if there exists $\vec s \in \partial I(\vec x^k)$ such that $\|\vec s\|_1>r^k$.

Suppose $\vec x^k = (x_1^k,\ldots,x_n^k)$, $\vec s = (s_1,\ldots,s_n)$, and then we  have 
\begin{align}
\lambda_i(t) & =|t|-\frac{x_i^k}{\|\vec x^k\|_\infty} t\ge 0, \,\,\,\forall\, i\in\{1,\ldots,n\},  \label{eq:lambda_i}\\
\|\vec s\|_1-r^k &=\|\vec s\|_1-\frac{(\vec x^k, \vec s)}{\|\vec x^k\|_\infty} = \sum_{i=1}^n  \lambda_i(s_i),
\end{align}
which yields 
\begin{equation}\label{eq:lambdage0}
\|\vec s\|_1-r^k>0\Leftrightarrow \exists\;\; i\in\{1,\ldots,n\}, \,\, s.t. \,\,\, \lambda_i(s_i)>0.
\end{equation}

Given $\vec x\in \mathbb{R}^n\setminus\{\vec 0\}$, 
let 
\begin{align}
\vec q &=(q _1,\ldots,q _n)\in (\mathbb{R}^+)^n, \quad q_i = \sum_{j:\{i,j\}\in E} w_{ij}\vec 1_{x_i=x_j},\\
\vec p &= (p_1,\ldots,p_n), \quad p_i = \sum_{j:\{i,j\}\in E} w_{ij}\sign(x_i-x_j)-q_i.
\end{align}
It can be readily verified that $\vec p\in \partial I(\vec x)$,
and 
\begin{equation}\label{eq:spq}
s_i\in(\partial I(\vec x))_i = [p_i-q_i,p_i+q_i],  \quad  \forall\, \vec s=(s_1,\ldots, s_n)\in\partial I(\vec x).
\end{equation}
Here $(\partial I(\vec x))_i$ denotes the projected interval of the convex domain $\partial I(\vec x)$ onto the $i$-th coordinate.


Denote $\bar{\vec p}  = (\bar p_1,\ldots,\bar p_n)$ by
\begin{equation}\label{eq:s_barte}
\bar p_i=
\begin{cases}
p_i\mp q_i,&\text{if $i\in S^{\pm}(\vec x)$},\\
p_i+\sign(p_i)q_i,&\text{if $i\in S^<(\vec x)$},
\end{cases}
\end{equation}
which is located on the boundary of $(\partial I(\vec x))_i = [p_i-q_i,p_i+q_i]$.

Accordingly, combining Eqs.~\eqref{eq:lambdage0}, \eqref{eq:spq} and the convexity of $\lambda_i(t)$ given in Eq.~\eqref{eq:lambda_i} leads to
\begin{align*}
& \exists\;\;\vec s\in \partial I(\vec x^k), \,\,\,\,\,\,\,\,\,\,\, s.t.\,\,\,\,\,\|\vec s\|_1>r^k \\
\Leftrightarrow \quad &\exists\;\; i\in\{1,\ldots,n\}, \,\,\,\, s.t. \,\,\, \max_{s_i\in(\partial I(\vec x))_i}\lambda_i(s_i)>0\\
\Leftrightarrow\quad &\exists\;\; i\in\{1,\ldots,n\}, \,\,\,\, s.t. \,\,\,\,\, \lambda_i(\bar p_i)>0.
\end{align*}
That is, if there exists $i\in\{1,\ldots,n\}$ such that $\lambda_i(\bar p_i)>0$, then there should  exist a subgradient $\vec s\in\partial I(\vec x^k)$ satisfying $s_i=\bar p_i$ and  $\|\vec s\|_1>r^k$. 
Hence, if $\bar{\vec p}\in \partial I(\vec x)$ (but hardly holds in general), then we may directly select $\bar{\vec p}$; 
otherwise  we are able to use $\bar{\vec p}$ as an indicator for the subgradient selection. 
%

Define a {\em partial  order} relation `$\le$' on $\mathbb{R}^2$ by $(x_1,y_1)\le (x_2,y_2)$ if and only if either
$x_1<x_2$ or $x_1=x_2, y_1\le y_2$ holds. 

Let $\Sigma(\vec x)$ be a collection of  permutations of $\set{1,2,\ldots,n}$ such that for any $\sigma\in\Sigma(\vec x)$,  it holds 
\begin{equation}\label{eq:por}
(x_{\sigma(1)},\bar{p}_{\sigma(1)})\leq(x_{\sigma(2)},\bar{p}_{\sigma(2)})\leq \cdots\leq (x_{\sigma(n)},\bar{p}_{\sigma(n)}).
\end{equation}

For any $\sigma\in \Sigma(\vec x)$, we select 
\begin{align}
\vec s^\sigma &= (s_1^\sigma,s_2^\sigma,\ldots,s_n^\sigma)\in \partial I(\vec x), \\
s_i^\sigma &= \sum_{j:\{i,j\}\in E}w_{ij} z_{ij}, \quad i=1,2,\ldots,n,  \\
z_{ij} &= \sign(\sigma^{-1}(i)-\sigma^{-1}(j))\in \sgn(x_i -x_j), \quad i,j = 1,2,\ldots,n.
\end{align}
The following theorem demonstrates that 
such subgradient selection is sufficient to guarantee the strict monotonicity $r^{k+1} > r^{k}$ (if any)
with the help of the statement (1) of Lemma \ref{lem:rs}.


\begin{theorem}\label{prop:S123}
For any permutation $\sigma\in \Sigma(\vec x^k)$, we have 
$\|\vec s^\sigma\|>r^k$ if and only if 
there exists $\vec s \in \partial I(\vec x^k)$ such that $\|\vec s\|>r^k$.
\end{theorem}

\begin{proof} The necessity is obvious, we only need to prove the sufficiency. 

Suppose $\vec s=(s_1,\ldots,s_n)\in \partial I(\vec x^k)$ satisfies $\|\vec s\|>r^k$.
Then it can be derived that $\vec x^k/ \norm{\vec x^k}\notin \sgn(\vec s)$ using the statement (1) of Lemma \ref{lem:rs}, i.e.,  there exists an $i_0\in \set{1,2,\cdots,n}$ such that  
\begin{equation}\label{xikni}
x_{i_0}^k/\norm{\vec x^k}\notin \sgn(s_{i_0}). 
\end{equation}

Denote $J=\{j|\,\,\,x_j^k=x_{i_0}^k\}$. 
Let $j_1,j_2\in J$ be the indexes minimizing and maximizing function $\sigma^{-1}(\cdot)$ over $J$, respectively. Then we have 
\begin{align*}
s^\sigma_{j_1}=&  \sum_{t:\{t,j_1\}\in E} w_{j_1t}\sign(\sigma^{-1}(j_1)-\sigma^{-1}(t))=p_{j_1}-q_{j_1},\\
s^\sigma_{j_2}=&  \sum_{t:\{t,j_2\}\in E} w_{j_2t}\sign(\sigma^{-1}(j_2)-\sigma^{-1}(t))=p_{j_2}+q_{j_2}.
\end{align*}

We claim that either $j_1$ or $j_2$ is an integer $j$ such that $x_j^k\notin \norm{\vec x^k}\sgn(s^\sigma_j)$. This directly yields $\|\vec s^\sigma\|>r^k$ according to the statement (1) of Lemma~\ref{lem:rs}
and thus completes the proof. Suppose the contrary:  
\begin{equation}\label{eq:xikeqxjk}
x_{i_0}^k=x_j^k\in \norm{\vec x^k}\sgn(s^\sigma_j),\,\,\,j=j_1,j_2.
\end{equation}
and split the discussion into the following two cases.

(1) $i_0\in S^<(\vec x^k)$.  

Eq.~\eqref{eq:xikeqxjk} implies that 
$$
s^\sigma_{j_1}=s^\sigma_{j_2}=0 \;\;\; \Leftrightarrow \;\;\; p_{j_1}=q_{j_1}\ge 0, \,\,\,p_{j_2}=-q_{j_2}\le 0,
$$
and thus, from Eq.~\eqref{eq:s_barte}, we have
\begin{align*}
\bar{p}_{j_1}&=p_{j_1}+\sign(p_{j_1})q_{j_1}=p_{j_1}+\sign(p_{j_1})p_{j_1}\ge0,\\
\bar{p}_{j_2}&=p_{j_2}+\sign(p_{j_2})q_{j_2}=p_{j_2}-\sign(p_{j_2})p_{j_2}\le0.
\end{align*}
Since $(x_{j_1},\bar{p}_{j_1})\le (x_{i_0},\bar{p}_{i_0})\le(x_{j_2},\bar{p}_{j_2})$ and $x_{j_1}=x_{i_0}=x_{j_2}$, it yields
\[
0\ge\bar{p}_{j_2}\ge\bar{p}_{i_0}\ge \bar{p}_{j_1}\ge 0\;\;\; \Rightarrow\;\;\; \bar{p}_{i_0} = 0
\;\;\; \Rightarrow\;\;\; p_{i_0}=q_{i_0}=0 \;\;\; \Rightarrow\;\;\; s_{i_0}=0, 
\]
which contradicts Eq.~\eqref{xikni}.

(2) $i_0\in S^\pm(\vec x^k)$. 

Eq.~\eqref{eq:xikeqxjk} implies that  
\begin{equation}
\begin{cases}
s_{i_0} \ge p_{i_0}-q_{i_0}=\bar{p}_{i_0}\ge \bar{p}_{j_1} =s^\sigma_{j_1}\ge 0,\text{ if } i_0\in S^+(\vec x^k),\\
s_{i_0} \le p_{i_0}+q_{i_0}=\bar{p}_{i_0}\le \bar{p}_{j_2} =s^\sigma_{j_2}\le 0,\text{ if } i_0\in S^-(\vec x^k),
\end{cases}
\end{equation}
both of which contradict Eq.~\eqref{xikni}.
\end{proof}


In a word, we choose $\vec s^k = \vec s^\sigma \in \partial I(\vec x^k)$ with $\sigma\in\Sigma(\vec x^k)$ in the simple iterative scheme \eqref{iter1}. Besides the above-mentioned strict increasing, we are able to show below that such subgradient selection assures finite-step local convergence.
Before that, we would like to further specify the choice of $\vec x^{k+1}$ from the closed solution set $X_p^{k+1}$ at the first step. There is a natural isomorphism $h$ by a central projection between $S_p$ and $S_\infty$ which are the unit spheres in norms $p$ and $\infty$, respectively. Denote $\partial X_p^{k+1}\in X_p^{k+1}$ be {\em the collection of points corresponding to the vertices of $h(X_p^{k+1})$ which is a convex polytope},  in view of the fact that the convex function $I(\vec x)$ achieves its maximum values on  vertices among $h(X_p^{k+1})$.
Hence, the three-step iterative scheme \eqref{iter1} can be crystallized into
%
%
%
%
%
%
%
\begin{subequations}
\label{iter1c}
\begin{numcases}{}
\vec x^{k+1} \in \partial X_p^{k+1}, \label{eq:twostep_x2c}
\\
r^{k+1}=F(\vec x^{k+1}),
\label{eq:twostep_r2c}
\\
\vec s^{k+1} = \vec s^\sigma, \,\,\, \sigma\in \Sigma(\vec x^{k+1}).
\label{eq:twostep_s2c}
\end{numcases}
\end{subequations}

Let 
\begin{equation}\label{eq:C}
C = \{\vec x\in\mathbb{R}^n \big| F(\vec y)\leq F(\vec x), \,\, \forall\,\vec y\in\{T_i\vec x: i\in\{1,\ldots,n\}\}\},
\end{equation}
where $T_i\vec x$ is defined as
\begin{equation}\label{eq:Tialp}
(T_i\vec x)_j =
\begin{cases}
x_j, & j\ne i,\\
-x_j, & j=i.
\end{cases}
\end{equation}


%

\begin{theorem}[finite-step local convergence]
\label{thm:conver_3}
Assume the sequences $\{\vec x^k\}$ and $\{r^k\}$ are generated by the simple iterative scheme \eqref{iter1c} from any initial point $\vec x^0\in \mathbb{R}^n\setminus \{\vec0\}$. There must exist $N\in\mathbb{Z}^+$ and $r^*\in\mathbb{R}$ such that, for any $k>N$, 
$r^k=r^*$, and $\vec x^{k+1}\in C$ are local maximizers.

\end{theorem}

\begin{proof}
According to Eq.~\eqref{eq:rsr}, 
for any integer $k>0$,
there exists   $\iota\subset \{1,\ldots,n\}$ such that
$r^k\le \sum_{i\in \iota}|s_i^k|\le r^{k+1}$, which means the sequence $\{r^k\}$ can only take finite values because 
the set 
$$
\left\{\sum_{i\in \iota}|s_i|\Big|\,\,\,\vec s^\sigma=(s_1,\ldots,s_n),\,\,\,\forall\,\sigma\in \Sigma(\vec x),\,\,\, \forall\, \vec x,\iota\right\}
$$ is finite. 
Thus there exist $N\in\mathbb{Z}^+$ and $r^*\in\mathbb{R}$ such that $r^k=r^*$ for any $k>N$, thereby implying that 
\begin{align*}
r^k &=\|\vec s^k\|_1, \\
X_p^{k+1}&=\{\vec x\,\big|\,\vec x/ \|\vec x\|_\infty\in\sgn(\vec s^k),\,\,\,\, \|\vec x\|_p = 1\},
\end{align*}
by virtue of Lemma~\ref{lem:rs}. 
That is, $\forall\,\vec x^{k+1}\in\partial X_p^{k+1}$, we have
\begin{align}
\vec x^{k+1}/ \|\vec x^{k+1}\|_\infty&\in\sgn(\vec s), \quad \forall\, s\in\partial I(\vec x^{k+1}), 
 \label{eq:x1}
\\
S^<(\vec x^{k+1})&=\varnothing. \label{eq:x2}
\end{align}
%
%
%

Now we will show $\vec x^{k+1}\in C$ and neglect the superscript $k+1$ hereafter for simplicity. 
Suppose the contrary that there exists $i\in S^\pm(\vec x)$ satisfying $F(T_i\vec x)>F(\vec x)$,
and then we have 
\begin{equation}
\|T_i\vec x\|_\infty=\|\vec x\|_\infty, \quad 
I(T_i\vec x)-I(\vec x)  = \pm\sum_{j:\{j,i\}\in E} w_{ij}2 x_j>0. 
\end{equation} 
Let $\vec s=(s_1,\ldots,s_n)\in\partial I(\vec x)$ be generated by
$$
z_{ij} =-x_j/\norm{\vec x}\in \sgn(x_i-x_j),
$$
and thus 
\begin{equation*}
x_is_i=x_i\sum_{j:\{j, i\}\in E} w_{ij}z_{ij} = -\sum_{j:\{j, i\}\in E} w_{ij}x_ix_j/\norm{\vec x}=-\frac{1}{2}(I(T_i\vec x)-I(\vec x))<0,
\end{equation*}
which contradicts Eq.~\eqref{eq:x1}.


Finally, we want to show $\vec x$ is a local maximizer of $F(\cdot)$ on $\mathbb{R}^n\setminus \{\vec0\}$.
Let $U$ be  a neighborhood of $\vec x$ such that  
$$
\|\frac{\vec y}{\norm{\vec y}}-\frac{\vec x}{\norm{\vec x}}\|_\infty<\frac12,\,\,\,\forall\,\vec y\in U,
$$ 
and 
$$
\vec y'=\frac{\norm{\vec x}}{\norm{\vec y}}  \vec y, \quad 
g(t)=I(t(\vec y'-\vec x)+\vec x), \quad \forall\,\vec y\in U.
$$
We claim 
\begin{equation}\label{eq:g}
g(t)-g(0)\leq 0, \quad \forall\,t\in[0,1],
\end{equation}
with which we are able to verify $\vec x$ is a local maximizer as follows 
\begin{align*}
F(\vec y)-F(\vec x)=&F(\vec y')-F(\vec x)=\frac{1}{\|\vec x\|_\infty}(I(\vec y')-I(\vec x))\\
=&\frac{1}{\|\vec x\|_\infty}(g(1)-g(0))\le 0,
\quad \forall\,\vec y\in U.
\end{align*}

The only remaining thing is to prove Eq.~\eqref{eq:g}. 
First, it can be easily shown that $g(t)$ is linear on $[0,1]$, and there exists $\vec s=(s_1,\ldots,s_n) \in\partial I(\vec x)$ such that its slope can be determined by 
$$
g(t)-g(0)=t(\vec s,\vec y'-\vec x).
$$
Then, the verification of Eq.~\eqref{eq:g} can be completed by 
$$
x_j (y'_j-x_j)\le 0\Rightarrow s_j(y'_j-x_j)\le 0, \quad  j=1,2,\ldots,n, 
$$
where we have used $\|\vec y'\|_\infty=\|\vec x\|_\infty$ as well as Eqs.~\eqref{eq:x1} and \eqref{eq:x2}. 
%
\end{proof}
\begin{remark}
Notice that $\{T_i\}_{i=1}^n$ can generate a commutative group $\mathcal{T}$ on $\mathbb{R}^n$. 
If we further restrict its domain to be $\{\vec x\,\,\big|S^<(\vec x)=\varnothing\}$, then $\vec x$ is a global maximizer of  $F(\cdot)$ on $\mathbb{R}^n\setminus \{\vec0\}$ if and only if $F(\vec y)\le F(\vec x)$ holds for any $\vec y=T\vec x$, $\forall\,T\in\mathcal{T}$. In such sense, the set $C$ given in \eqref{eq:C} is the first order approximation to global maximizers. 
\end{remark}


\section{Numerical experiments}
\label{sec:Numer-exper}

In this section, we conduct performance evaluation of the proposed SI algorithm \eqref{iter1c} on the graphs with positive weight in G-set (available from: http://www.stanford.edu/yyye/yyye/Gset), and
always set the initial data $\vec x^0$ to be the maximal eigenvector of the graph Laplacian \cite{DelormePoljak1993,PoljakRendl1995}. The three bipartite graphs $G48$, $G49$, $G50$ will not be considered because their optima cuts can be achieved at the initial step.
The {recently updated cut values achieved by a multiple search operator heuristic} are chosen to be the reference \cite{MaHao2017}.
SI is capable of producing approximate cuts with high quality: 
the ratios between the resulting cut values and the reference ones are at least  $0.986$ (see Tab.~\ref{tab1}),
and can be improved to $0.997$ (see Tab.~\ref{tab:adl}) after introducing a straightforward perturbation to break out of local optima.



%

\subsection{Implementation and cost analysis}

We begin with the algorithm implementation plus a preliminary cost analysis.
Alg.~\ref{alg1} gives the pseudo-code of the SI algorithm \eqref{iter1c},
where $\vec x^{k+1}$, $r^{k+1}$ and $\vec s^{k+1}$ are 
generated in Line~\ref{alg:es},  Line~\ref{alg:F} and Lines~\ref{alg:q}-\ref{alg:sg}, respectively.
In Line~\ref{alg:es}, the function $\Call{exact\_solution}{}$ randomly selects the iteration point $\vec x^{k+1}\in \partial X_p^{k+1}$ where the exact solution set $ X_p^{k+1}$ is given in Theorem \ref{Thm:exact_solution}.
After $\bar{\vec p}$ is obtained in Line~\ref{alg:pbar} using Eq.~\eqref{eq:s_barte},
we are able to arrive at the partial order \eqref{eq:por} through the Bubble Sort procedure
during which the requested subgradient $\vec s^{k+1}$ can be automatically updated in an iterative manner (see Lines~\ref{alg1:s1} and \ref{alg1:s2}). 
This constitutes the subroutine named by $\Call{subgradient}{}$ which starts from Line~\ref{alg:sgs}.
It should be pointed out that there is randomness in determining both $\vec x^{k+1}$ and $\vec s^{k+1}$,
so that the performance evaluation below is conducted in the sense of average
by re-running SI \eqref{iter1c} 100 times from the same initial data.
A preliminary estimate of the cost for three iteration steps $T=50,500,2000$
is presented in Tab.~\ref{tab:epsilon} where we have set $p=2$ for instance.


%


The main task for reaching the exact solution set $X_p^{k+1}$ by Theorem \ref{Thm:exact_solution} is to determine $m_0$ 
with which $m_1$ can be automatically obtained by Eq.~\eqref{eq:order2}.
In our implementation, $\Call{exact\_solution}{}$ uses the Bubble Sort to arrange $\vec s^k$
in an ascending order such that the accumulation of increment $A(n),A(n-1),\ldots,A(1)$ is calculated, successively, until $m_0(k)$ (i.e., $m_0$ at the $k$-step iteration) can be determined via Eq.~\eqref{eq:defm0}. The cost of this procedure is $\mathcal{O}((n-m_0(k))n)$.
Tab.~\ref{tab:epsilon} shows that the mean values of $(n-m_0(k))/n$ are far less that $1$,
which are at most  $0.060,0.048,0.047$ for $T=50,500,2000$, respectively.
In particular, $\Call{exact\_solution}{}$ reduces to {\bf Scenario 3} and thus only costs $\mathcal{O}(n)$ for $p=\infty$.

The subroutine $\Call{subgradient}{}$ is also a Bubble Sort procedure and costs $\mathcal{O}(n^2)$ in worst cases. However, it should be pointed out that the efficiency of Bubble Sort depends on the initial order and actually costs  $\mathcal{O}((\delta_\sigma(k)+1)n)$, where 
\begin{equation}\label{eq:epsilonk} 
\delta_\sigma(k)=\frac{\sum_{i=1}^n|\sigma^{k+1}(i)-\sigma^{k}(i)|}{2n}
\end{equation}
denotes an average displacement of the permutation $\sigma$ used in Eq.~\eqref{eq:por} for the $k$-th iteration. 
Tab.~\ref{tab:epsilon} shows that the ratios between $\mean(\delta_\sigma(k))$ and $n$ are at most $0.03$.


It remains to estimate the cost for updating $r,\vec q,\vec p$ in Lines~\ref{alg:F}-\ref{alg:p},
where we prefer to only calculate the increment: 
\begin{equation}\label{eq:deltag}
\delta_\Gamma(\vec x^{k+1},\vec x^k):=\Gamma(\vec x^{k+1})-\Gamma(\vec x^k),\,\,\,\Gamma\in \{F,\vec q,\vec p\}.
\end{equation}
 Let 
\begin{equation}
\vec z^k= \frac{\vec x^k}{\norm{\vec x^k}},\,\,\,
\vec z^{k,i}=(z^{k,i}_1,\ldots,z^{k,i}_n) = (z_1^{k+1},\ldots, z_i^{k+1},z_{i+1}^k,\ldots, z_{n}^k).
\end{equation}
Then we have
\begin{align}
\delta_\Gamma(\vec x^{k+1},\vec x^k) &=\delta_\Gamma(\vec z^{k+1},\vec z^k)=\sum_{i=1}^n\delta_\Gamma(\vec z^{k,i},\vec z^{k,i-1})\\
&=\sum_{i\in V(k)}\delta_\Gamma(\vec z^{k,i},\vec z^{k,i-1}),\;\;\;\Gamma\in\{F,\vec q,\vec p\}\label{eq:deltagamma}
\end{align}
where
\begin{equation}
\label{eq:Vk}
V(k)= \{i\,\,\big|\,\,\, z_i^{k+1}\ne z_i^k\}.
\end{equation}
Consequently, the complexity of calculating $\delta_\Gamma(\vec x^{k+1},\vec x^k)$ is $\mathcal{O}(|V(k)| n)$ 
since we need $\mathcal{O}( n)$ to produce each component $\delta_\Gamma(\vec z^{k,i},\vec z^{k,i-1})$.
Tab.~\ref{tab:epsilon} reveals that $\mean(|V(k)|)$ is much smaller than $n$ and sometimes 
vanishes as $T$ increases. 

 
In total, the SI algorithm costs $\mathcal{O}( c(k) n)$ in the $k$-th step and $c(k) := (n-m_0(k))+\delta_\sigma(k)+|V(k)|$
depends on the underlying graph like its weights and order. Actual numerical experiments in Tab.~\ref{tab:epsilon} show that the mean value of $c(k)$ is much smaller than $n$. 

\subsection{Quality check}
\label{sec:Qc}


We are now ready for quality check of numerical solutions achieved by the simple algorithm \eqref{iter1c}.
The numerical results for the RSC algorithm based on graph Laplacian ($\Delta_2$-RSC) \cite{th:Ottaviano2008} and graph $1$-Laplacian ($\Delta_1$-RSC) \cite{ChangShaoZhang2016-maxcut}, as well as the GW algorithm \cite{th:Ottaviano2008}
are adopted for comparison.


The quality check is performed based on numerical solutions until $T=2000$. 
Tab.~\ref{tab1} shows the minimum, mean and maximum cut values during $100$ runs for $p=1,2,\infty$.
It can be easily seen there that the results for different  $p$ $(=1,2,\infty)$ are comparable and
are all very close to the reference values.
Actually, the ratios between the best cut values by SI (chosen from the maximum cut values over $p=1,2,\infty$) and the reference ones
are at least  $0.986$ (see the results for G36),
while the numerical lower bound for such ratios is about $0.946$,
$0.933$ and $0.949$ for the GW, $\Delta_2$-RSC
and $\Delta_1$-RSC algorithms, respectively. 
In particular, for the case of $p=\infty$, the ratios between the minimum, mean, maximum cut values and the reference ones
are at least  $0.979$, $0.982$, $0.986$, respectively,
all of which are larger than the average ratios over these $27$ graphs
for the $\Delta_1$-RSC ($\simeq 0.971$), GW ($\simeq 0.960$), $\Delta_2$-RSC ($\simeq 0.958$), 
and SC ($\simeq 0.951$) algorithms.
The SC cuts are obtained by rounding the initial data with a threshold of $0$ and the cut values are shown in the second column of Tab.~\ref{tab1}.
In order to further show the overall performance in achieving high ratio by the SI algorithm, 
we plot the histogram of the ratios for all $3\times2700$ runs in Fig.~\ref{fig2}.
We are able to clearly observe there that, 
(1) more than $95\%$ of runs achieve ratios exceeding $0.986$ (see the black vertical line of Fig.~\ref{fig2}); 
(2) the percent of runs obtained a ratio larger than $0.986$ exceeds $72\%$.

%
%
%
%
 

 \subsection{Breaking out of local optima}
\label{sec:mjlo}



Within the SI algorithm, we are allowed to plug into local breakout techniques to further improve the solution quality. 
A preliminary attempt is to generate a new point $\tilde{\vec x}^{k+1}=(\tilde x_1,\ldots,\tilde x_n)$ in a stochastic  manner: 
 \begin{equation}\label{eq:pertu}
 \tilde x_i=
 \begin{cases}
 -x_i^k,\text{ with the probability of $e^{-\beta|\bar p^k_i|}$,}\\
 x_i^k,\text{ with the probability of $1-e^{-\beta|\bar p^k_i|}$,}
 \end{cases}
 \end{equation}
 when SI is stuck at $\vec x^k=(x_1^k,\ldots,x_n^k)$,
 namely, it cannot make the function $F(\cdot)$ increase. 
 Here we choose $\beta$ to be a controllable parameter, 
 and $\bar{\vec p}^k = (\bar p_1^k,\ldots,\bar p_n^k)$ to generate $\tilde{\vec x}^{k+1}$ in order to decrease $F(\cdot)$ (if any) as little as possible 
 in view of the following fact:
\begin{equation}
F(T_i\vec x^k)-F(\vec x^k)=-|\bar p_i^k|, \quad i=1,\ldots,n,
\end{equation}
where $T_i$ is defined in Eq.~\eqref{eq:Tialp}. 
In such sense, we regard the manner \eqref{eq:pertu} as a special kind of perturbation,
and the resulting algorithm is named by the simple iteration with perturbation (SI-P).

Alg.~\ref{alg:ipwp} presents the pseudo-code of SI-P. In Lines~\ref{alg:ips}-\ref{alg:ipe}, 
$\Call{si\_perturb}{}$ undergoes the simple iteration equipped with jumping out of local optima until a given final time $T$
during which the perturbation \eqref{eq:pertu} triggers with a prescribed $\beta$ if the function values remain unchanged for $t$ steps (see Lines~\ref{trigger0}-\ref{trigger1}). It should be noted that $\Call{si\_perturb}{}$ returns the maximal cut value $r_{opt}$ and the corresponding point $\vec x_{opt}$ during the period of $T$. 
We are now left only to choose $\beta$ which should not only depend on the iteration time but also have a specific range because $\tilde{\vec x}^{k+1} = -\vec x^k$ when $\beta=0$ and $\tilde{\vec x}^{k+1} = \vec x^k$ when $\beta\to \infty$.
Here Alg.~\ref{alg:ipwp} adopts a naive way via an outer loop in Lines~\ref{alg:whiles}-\ref{alg:whilee},
and takes $L$ runs of $\Call{si\_perturb}{}$ with different $\beta$ randomly chosen from $(0,1)$ 
within each turn of loop. This outer loop continues until the cut value stops increasing (see Lines~\ref{break0}-\ref{alg:loope}) and the variable $count$ records the total number of turns (see Line~\ref{code:count}).


Tab.~\ref{tab:adl} shows the numerical results by SI-P with $t=3$, $T=2000$, $L=20$, and $p=\infty$. 
Now the cut values are all increased for those $27$ problem instances in G-set
and the ratios between the best cut values and {the reference ones}
become at least  $0.997$. Moreover, 
the complexity of the function $\Call{si\_perturb}{}$ is almost the same as SI. 
Therefore, SI-P calls $\Call{si\_perturb}{}$ $count\times L$ times and performs $count\times L\times T$ iterations in total.
This is the price we should pay for the improved cut values,
which is at most $17 \times 20\times 2000=680000$ iteration steps in the numerical experiments on G-set
(see the last column of Tab.~\ref{tab:adl}).



\subsection{Comparison with CirCut}
\label{sec:circut}

Finally, we compare SI in Alg.~\ref{alg1} and $\Call{SI\_PERTURB}{}$ in Alg.~\ref{alg:ipwp} with the primal CirCut algorithm given in Alg.~1 of \cite{BurerMonteiroZhang2001}
which does not invoke any local search techniques. For convenience, we adopt the same notations as those used in \cite{BurerMonteiroZhang2001} unless otherwise specified. It should be noted that CirCut 
do need a simple gradient algorithm with a backtracking Armijo line-search to minimize the nonconvex objective function $f(\theta)$  from $\theta^0$ instead of calling an external solver. The stopping condition for the line-search is,  
either the relative change in $f(\theta)$ is less than $\epsilon_f$ or the relative change in its gradient is less than $\epsilon_g$.
We use $T_f$ (resp. $T_g$) to count the total number of times $f(\theta)$ (resp. $\nabla f(\theta)$) has been evaluated during the line-search. 
CirCut stops until $N$ consecutive random perturbations cannot improve the approximate cut, and let $T_P$ count the total number of perturbations. For a fair comparison, we extract the first two steps: line-search and Procedure-CUT, to form a pure rank-two relaxation for maxcut, abbreviated as CirCut0, which excludes the perturbation step,  and set it against SI. When the random perturbation is invoked,  CirCut is compared with $\Call{SI\_PERTURB}{}$.  For a given graph,  one usually runs the algorithm $M$ times with multiple random starting points: $\theta^0\sim \mathcal{U}(0,2\pi)$, and $\bar{T}_\gamma$ gives the average value of $T_\gamma$ over these $M$ times for $\gamma\in \{f, g, P\}$. 
The iteration of SI is not stoped until the cut values remain unchanged for $t$ consecutive steps and we run $\Call{SI\_PERTURB}{}$  $\bar{T}_P$ times staring from a single initial data where $T_S$ counts the total number of iterations. Both SI and $\Call{SI\_PERTURB}{}$  are re-run $M$ times from the same initial data given by the maximal eigenvector of the graph Laplacian where $\bar{T}_S$ denotes the average value of $T_S$. 

We set $N=10$, $M=20$, $t=3$ and use the implementation of CirCut available at Github\footnote{see {\tt https://github.com/MQLib/MQLib/tree/master/src/heuristics/maxcut}} where the parameters for the line-search are: the maximum number of rounds $n_{\max}=200$,  and the tolerances $\epsilon_f=\epsilon_g=1$e-4. Fig.~\ref{fig3} plots the minimum, mean and maximum cut values (normalized by the number of edges $|E|$) from the simulations with the $M$ starting points. 
We are able to observe there that the quality of SI solutions is much better than CirCut0,
and the approximate cuts produced by SI$\_$PERTURB are of comparable quality to CirCut with the same number of perturbations.  

According to Theorem~\ref{thm:conver_2} and Theorem~\ref{thm:conver_3}, the cut values obtained by SI are monotonically 
updated, and the iterative solution  converges to a  local optimum from any given initial data. 
So it will be interesting to see whether SI can improve the output cuts of CirCut0 and the results are displayed in Fig.~\ref{fig3}(a) with dot lines and legend ``CirCut0+SI'' where the maximum number of SI iterations is set to be $100$. It is shown there that SI does improve the quality of solutions obtained by CirCut0, 
which clearly indicates that CirCut0 is not guaranteed to get local optimum. 
Almost the same story happens with CirCut+SI (see Fig.~\ref{fig3}(b)).
On the other hand, we should point out that CirCut0 cannot improve the quality of the solution produced by SI. In fact, the solution obtained by SI must be a cut, while, by Theorem 3.4 in \cite{BurerMonteiroZhang2001}, any cut  corresponds to a critical  point of $f(\theta)$, which means that the stopping condition  $\nabla f(\theta)\approx \vec0$ in the line-search is immediately satisfied, and there is obviously no way to improve it further.

\blue{We run all above-mentioned algorithms in Matlab (r2019b) on a High-Performance Computing Platform: 2*Intel Xeon E5-2650-v4 (2.2GHz, 30MB Cache, 9.6GT/s QPI Speed, 12 Cores, 24 Threads) with 128GB Memory. In Tab.~\ref{tab:circutcost}, we list the values of $\bar{T}_\gamma$ with $\gamma\in \{f, g, S, P\}$ as well as the average  wall-clock time in seconds needed by only one thread for each run}. In each step of the line-search, it is required to calculate $f(\theta)$ and its gradient 
the complexity of which is \blue{$\mathcal{O}(n^2)$} and does not change significantly as the search proceeds.
\blue{Then, it can be deduced that the run time of both CirCut0 and CirCut should be roughly proportional to $\bar{T}_f + \bar{T}_g$, which can be readily verified in Tab.~\ref{tab:circutcost}.} 
By comparison, thanks to its monotonicity in Theorem~\ref{thm:conver_2} and local adjustability in  Theorem~\ref{thm:conver_3}, the complexity of one iteration step decreases as SI goes on,
and its average over all steps is \blue{$\mathcal{O}(\mean(c(k) )n)$ and usually much less than $\mathcal{O}(n^2)$ as already shown in Tab.~\ref{tab:epsilon}}. \blue{There are two important implications. One is that the time ratio between SI and CirCut0 (or CirCut) should be roughly proportional to $\bar{T}_S/(\bar{T}_f + \bar{T}_g)$ where $\bar{T}_S$ for SI is not so large (less than $50$).
This can be easily checked in Tab.~\ref{tab:circutcost}. The other is the run time of SI$\_$PERTURB should grow much more slowly as the iteration goes on where $\bar{T}_S$ is always larger than $500$,  
which can be seen by comparing the time between SI and SI$\_$PERTURB in Tab.~\ref{tab:circutcost}.
Hence we are able to tell that SI and SI$\_$PERTURB require much fewer iterations and thus run much faster than CirCut0 and CirCut, respectively.} That is, although both SI and CirCut do not need call any external solver, quickly locating and checking stationary points of $f(\theta)$ as many as possible in CirCut is not that simple.

\section{Conclusion and outlook}
\label{sec:conclusion}




An equivalent continuous fractional optimization problem and a simple iterative (SI) algorithm going from one cut to another in a monotonic and rounding-free way for the maxcut problem were proposed. `{\em Simple}' means SI utilizes the exact solutions of the inner subproblems. Numerical experiments on G-set demonstrated that the continuous SI algorithm can produce more qualified solutions than all other existing continuous algorithms. The underlying guiding thought is to build a firm bridge between discrete data world and continuous math field and then use it to design more efficient algorithms. Introducing more advanced combinatorial heuristics into SI and further improving the quality of solutions are on the way. Our attempts on the maxcut problem may provide a valuable reference for other combinatorial problems and fractional programming problems.

\clearpage

\begin{algorithm}

\caption{\small Pseudo-code of the simple iterative (SI) algorithm.}
\begin{algorithmic}[1]
\State Initialize $\vec x^0$, $r^0$, $\vec q^0$, $\vec p^0$, $\sigma^0$, $\vec s^0$, $T$
\State $k \gets 0$
\While{$k<T$}
\State $\vec x^{k+1}\gets \Call{exact\_solution}{\vec s^k,r^k}$\label{alg:es}
\State $r^{k+1} \gets r^k+\delta_F(\vec x^{k+1},\vec x^k)$\label{alg:F}
\State $\vec q^{k+1} \gets \vec q^k +\delta_{\vec q}(\vec x^{k+1},\vec x^k)$\label{alg:q}
\State $\vec p^{k+1} \gets \vec p^k + \delta_{\vec p}(\vec x^{k+1},\vec x^k)$\label{alg:p}
\State $\bar{\vec p}^{k+1}\gets \Call{p\_bar}{\vec x^{k+1},\vec q^{k+1},\vec p^{k+1}}$\label{alg:pbar}
\State $(\vec s^{k+1},\sigma^{k+1}) \gets \Call{subgradient}{\vec x^{k+1},\bar{\vec p}^{k+1},\vec s^{k},\sigma^{k}}$\label{alg:sg}
\State $k\gets k+1$
\EndWhile
\Function{$(\vec s,\sigma)$ =  subgradient}{$\vec x,\bar{\vec p},\vec s,\sigma$}\label{alg:sgs}
\For{$i=1$ to $n-1$}
\For{$j=i+1$ to $2$}
\If{$(\vec x(\sigma(j)),\bar{\vec p}(\sigma(j))) \le (\vec x(\sigma(j-1)),\bar{\vec p}(\sigma(j-1)))$}
\State $\sigma \gets \Call{swap}{\sigma,j-1,j}$
\State $\vec s(\sigma(j-1))\gets \vec s(\sigma(j-1))-2w_{\sigma(j)\sigma(j-1)}$ \label{alg1:s1}
\State $\vec s(\sigma(j))\gets \vec s(\sigma(j))+2w_{\sigma(j)\sigma(j-1)}$ \label{alg1:s2}
\Else ~break
\EndIf
\EndFor
\EndFor
\EndFunction \label{alg:sge}
\end{algorithmic}
\label{alg1}
\end{algorithm}

\clearpage

\begin{table}

\begin{adjustbox}{addcode={\begin{minipage}{\width}}{
\caption{\small Cost analysis:  Mean values of $(n-m_0^k)/n$,  $\delta_\sigma(k)/n$,
    and $|V(k)|/n$ obtained from $100$ runs of the simple iterative (SI) algorithm \eqref{iter1c} with $p=2$.  
    Here $(n-m_0^k)$,  $\delta_\sigma(k)$, and $|V(k)|$ are given in Eq.~\eqref{eq:defm0}, Eq.~\eqref{eq:epsilonk}, and Eq.~\eqref{eq:Vk}, respectively, $n$ denotes the size of graph and $T$ gives the total step of iterations. The time complexity of SI is $\mathcal{O}(\mean(c(k) )nT)$ with $c(k) = (n-m_0(k))+\delta_\sigma(k)+|V(k)|$ depending on the underlying graph.  The values of $\mean(c(k))/n$  are at most $0.095,0.077,0.077$ for $T=50,500,2000$, respectively,
thereby indicating $\mean(c(k))$ is much smaller than $n$. 
}\label{tab:epsilon}\end{minipage} },angle=90,center}
 \centering   
\begin{tabular}{|r|r|rrr|rrr|rrr|}
\hline
\multirow{2}*{\makecell{graph}}&  \multirow{2}*{\makecell{$n$} }& \multicolumn{3}{r|}{\makecell{$\mean(n-m_0^k)/n$}}& \multicolumn{3}{r|}{\makecell{$\mean(\delta_\sigma(k))/n$}}& \multicolumn{3}{r|}{\makecell{$ \mean(|V(k)|)/n$}}\\
\cline{3-11}
    && $T=50$& $T=500$& $T=2000$ & $T=50$&$T=500$& $T=2000$ &$T=50$&$T=500$& $T=2000$ \\
    \hline
    G1    & 800   & 0.0073  & 0.0046  & 0.0045  & 0.0091  & 0.0082  & 0.0083  & 0.0010  & 0.0000  & 0.0000  \\
    G2    & 800   & 0.0069  & 0.0043  & 0.0042  & 0.0090  & 0.0084  & 0.0082  & 0.0008  & 0.0000  & 0.0000  \\
    G3    & 800   & 0.0072  & 0.0048  & 0.0046  & 0.0093  & 0.0081  & 0.0081  & 0.0009  & 0.0000  & 0.0000  \\
    G4    & 800   & 0.0061  & 0.0042  & 0.0040  & 0.0090  & 0.0082  & 0.0081  & 0.0061  & 0.0042  & 0.0040  \\
    G5    & 800   & 0.0059  & 0.0045  & 0.0043  & 0.0089  & 0.0081  & 0.0079  & 0.0007  & 0.0000  & 0.0000  \\
    G14   & 800   & 0.0510  & 0.0423  & 0.0412  & 0.0246  & 0.0243  & 0.0251  & 0.0034  & 0.0000  & 0.0000  \\
    G15   & 800   & 0.0550  & 0.0444  & 0.0440  & 0.0248  & 0.0253  & 0.0247  & 0.0068  & 0.0000  & 0.0000  \\
    G16   & 800   & 0.0532  & 0.0445  & 0.0428  & 0.0229  & 0.0237  & 0.0237  & 0.0532  & 0.0445  & 0.0428  \\
    G17   & 800   & 0.0567  & 0.0471  & 0.0461  & 0.0244  & 0.0244  & 0.0244  & 0.0020  & 0.0000  & 0.0000  \\
    G22   & 2000  & 0.0146  & 0.0106  & 0.0102  & 0.0120  & 0.0125  & 0.0120  & 0.0018  & 0.0001  & 0.0001  \\
    G23   & 2000  & 0.0119  & 0.0100  & 0.0099  & 0.0123  & 0.0117  & 0.0118  & 0.0021  & 0.0000  & 0.0000  \\
    G24   & 2000  & 0.0135  & 0.0100  & 0.0097  & 0.0125  & 0.0117  & 0.0120  & 0.0135  & 0.0100  & 0.0097  \\
    G25   & 2000  & 0.0161  & 0.0115  & 0.0111  & 0.0132  & 0.0124  & 0.0122  & 0.0161  & 0.0115  & 0.0111  \\
    G26   & 2000  & 0.0140  & 0.0101  & 0.0097  & 0.0126  & 0.0127  & 0.0127  & 0.0038  & 0.0000  & 0.0000  \\
    G35   & 2000  & 0.0475  & 0.0380  & 0.0370  & 0.0274  & 0.0259  & 0.0255  & 0.0475  & 0.0380  & 0.0370  \\
    G36   & 2000  & 0.0566  & 0.0484  & 0.0466  & 0.0252  & 0.0286  & 0.0278  & 0.0054  & 0.0005  & 0.0000  \\
    G37   & 2000  & 0.0516  & 0.0427  & 0.0414  & 0.0260  & 0.0272  & 0.0271  & 0.0076  & 0.0000  & 0.0000  \\
    G38   & 2000  & 0.0595  & 0.0478  & 0.0466  & 0.0260  & 0.0292  & 0.0302  & 0.0091  & 0.0000  & 0.0000  \\
    G43   & 1000  & 0.0140  & 0.0102  & 0.0100  & 0.0131  & 0.0125  & 0.0122  & 0.0024  & 0.0000  & 0.0000  \\
    G44   & 1000  & 0.0169  & 0.0134  & 0.0133  & 0.0126  & 0.0135  & 0.0129  & 0.0013  & 0.0000  & 0.0000  \\
    G45   & 1000  & 0.0152  & 0.0137  & 0.0135  & 0.0126  & 0.0130  & 0.0132  & 0.0152  & 0.0137  & 0.0135  \\
    G46   & 1000  & 0.0112  & 0.0093  & 0.0091  & 0.0121  & 0.0121  & 0.0121  & 0.0011  & 0.0000  & 0.0000  \\
    G47   & 1000  & 0.0125  & 0.0103  & 0.0100  & 0.0125  & 0.0125  & 0.0122  & 0.0125  & 0.0103  & 0.0100  \\
    G51   & 1000  & 0.0527  & 0.0440  & 0.0427  & 0.0228  & 0.0241  & 0.0239  & 0.0016  & 0.0004  & 0.0000  \\
    G52   & 1000  & 0.0485  & 0.0406  & 0.0391  & 0.0203  & 0.0239  & 0.0235  & 0.0485  & 0.0406  & 0.0391  \\
    G53   & 1000  & 0.0544  & 0.0462  & 0.0453  & 0.0212  & 0.0242  & 0.0245  & 0.0031  & 0.0000  & 0.0000  \\
    G54   & 1000  & 0.0616  & 0.0544  & 0.0520  & 0.0215  & 0.0241  & 0.0235  & 0.0017  & 0.0000  & 0.0000  \\
    \hline
\end{tabular}%
\end{adjustbox}%
 \end{table}

\clearpage
 
 \begin{table}
\centering
\begin{adjustbox}{addcode={\begin{minipage}{\width}}{ 
\caption{\small  Quality check: Cut values for $27$ problem instances in G-set obtained from $100$ runs of the simple iterative (SI) algorithm \eqref{iter1c} with $p=1,2,\infty$. 
Each run starts from the maximal eigenvector of the graph Laplacian and undergoes $T=2000$ iterations. 
The minimum, mean and maximum cut values are recorded and
compared to {the reference ones} obtained by combinatorial algorithms \cite{MaHao2017},
while the initial cut values are listed in the second column. 
The ratios between the best cut values (in italics) by SI and {the reference values}
are at least  $0.986$ (see G36).}
\label{tab1}\end{minipage} },angle=90,center}
\begin{tabular}{|r|r|rrr|rrr|rrr|c|}
\hline

\multirow{2}*{\makecell{graph}} & \multirow{2}*{\makecell{initial}} &\multicolumn{3}{r|}{\makecell{$p=1$}}& \multicolumn{3}{r|}{\makecell{$p=2$}}& \multicolumn{3}{r|}{\makecell{$p=\infty$}}&  \multirow{2}*{\makecell{{reference}}}\\
\cline{3-11}
    &&min & mean& max& min & mean& max&min & mean& max& \\
\hline
    G1    & 11221 & 11477 & 11524.52 & 11554 & 11505 & 11527.27 & 11548 & 11499 & 11523.3 & 11553 & 11624 \\
    G2    & 11283 & 11483 & 11525.91 & 11583 & 11484 & 11505.6 & 11535 & 11473 & 11517.42 & 11557 & 11620 \\
    G3    & 11298 & 11497 & 11542.53 & 11599 & 11564 & 11586.14 & 11602 & 11557 & 11586.99 & 11605 & 11622 \\
    G4    & 11278 & 11520 & 11561.1 & 11597 & 11545 & 11564.49 & 11606 & 11539 & 11568.26 & 11598 & 11646 \\
    G5    & 11370 & 11530 & 11568.89 & 11602 & 11561 & 11579.89 & 11607 & 11558 & 11577.22 & 11602 & 11631 \\
    G14   & 2889  & 2998  & 3016.89 & 3036  & 3015  & 3022.74 & 3030  & 3012  & 3026.18 & 3033  & 3064 \\
    G15   & 2771  & 2968  & 2986.36 & 3005  & 2989  & 2995.56 & 3006  & 2985  & 2996.68 & 3008  & 3050 \\
    G16   & 2841  & 2973  & 2995.32 & 3009  & 3005  & 3010.14 & 3014  & 3005  & 3011.76 & 3017  & 3052 \\
    G17   & 2866  & 2963  & 2981.65 & 3004  & 2996  & 3002.27 & 3012  & 2995  & 3002.7 & 3010  & 3047 \\
    G22   & 12876 & 13198 & 13243.55 & 13302 & 13237 & 13267.7 & 13286 & 13248 & 13266.04 & 13290 & 13359 \\
    G23   & 12817 & 13173 & 13220.99 & 13267 & 13238 & 13252.85 & 13271 & 13232 & 13246.72 & 13264 & 13344 \\
    G24   & 12826 & 13194 & 13221.16 & 13244 & 13235 & 13260.34 & 13295 & 13233 & 13262.4 & 13287 & 13337 \\
    G25   & 12781 & 13155 & 13189.61 & 13245 & 13185 & 13221 & 13240 & 13187 & 13213.86 & 13238 & 13340 \\
    G26   & 12752 & 13140 & 13176.81 & 13222 & 13185 & 13201.6 & 13221 & 13183 & 13197.21 & 13210 & 13328 \\
    G35   & 7194  & 7512  & 7540.39 & 7572  & 7558  & 7576.68 & 7595  & 7560  & 7575.4 & 7586  & 7687 \\
    G36   & 7124  & 7502  & 7529.77 & 7557  & 7535  & 7548.22 & 7566  & 7531  & 7553.36 & 7576  & 7680 \\
    G37   & 7162  & 7505  & 7541.03 & 7563  & 7565  & 7581.08 & 7603  & 7575  & 7590.59 & 7602  & 7691 \\
    G38   & 7122  & 7513  & 7542.91 & 7568  & 7561  & 7573.63 & 7589  & 7544  & 7565.93 & 7593  & 7688 \\
    G43   & 6395  & 6570  & 6609.44 & 6635  & 6599  & 6625.47 & 6645  & 6604  & 6625.8 & 6644  & 6660 \\
    G44   & 6439  & 6575  & 6598.42 & 6618  & 6593  & 6608.63 & 6617  & 6590  & 6609.27 & 6621  & 6650 \\
    G45   & 6364  & 6574  & 6599.49 & 6624  & 6587  & 6595.39 & 6612  & 6586  & 6593.07 & 6599  & 6654 \\
    G46   & 6389  & 6557  & 6586.68 & 6612  & 6562  & 6583.31 & 6597  & 6569  & 6583.93 & 6602  & 6649 \\
    G47   & 6353  & 6552  & 6583.5 & 6614  & 6584  & 6598.8 & 6609  & 6584  & 6597.24 & 6610  & 6657 \\
    G51   & 3645  & 3769  & 3787.88 & 3810  & 3788  & 3797.76 & 3811  & 3785  & 3794.12 & 3806  & 3848 \\
    G52   & 3645  & 3766  & 3787.31 & 3807  & 3808  & 3811.55 & 3816  & 3806  & 3811.85 & 3820  & 3851 \\
    G53   & 3630  & 3778  & 3793.98 & 3812  & 3795  & 3804.37 & 3818  & 3797  & 3808.45 & 3817  & 3850 \\
    G54   & 3655  & 3767  & 3789.34 & 3805  & 3800  & 3809.05 & 3818  & 3808  & 3813.69 & 3821  & 3852 \\
    \hline
\end{tabular}
\end{adjustbox}

\end{table}

\clearpage

\begin{figure}[h]
  \centering
 \caption{\small Quality check: Histogram of the ratios for all $2700\times 3$ runs depicted in Tab.~\ref{tab1} 
 (More explanations are referred to Tab.~\ref{tab1}). 
 The percent of runs obtained a ratio larger than $0.980$, which lie on the right of the black vertical line,  
 exceeds $95\%$. }
 \label{fig2}
 \includegraphics[scale=0.25]{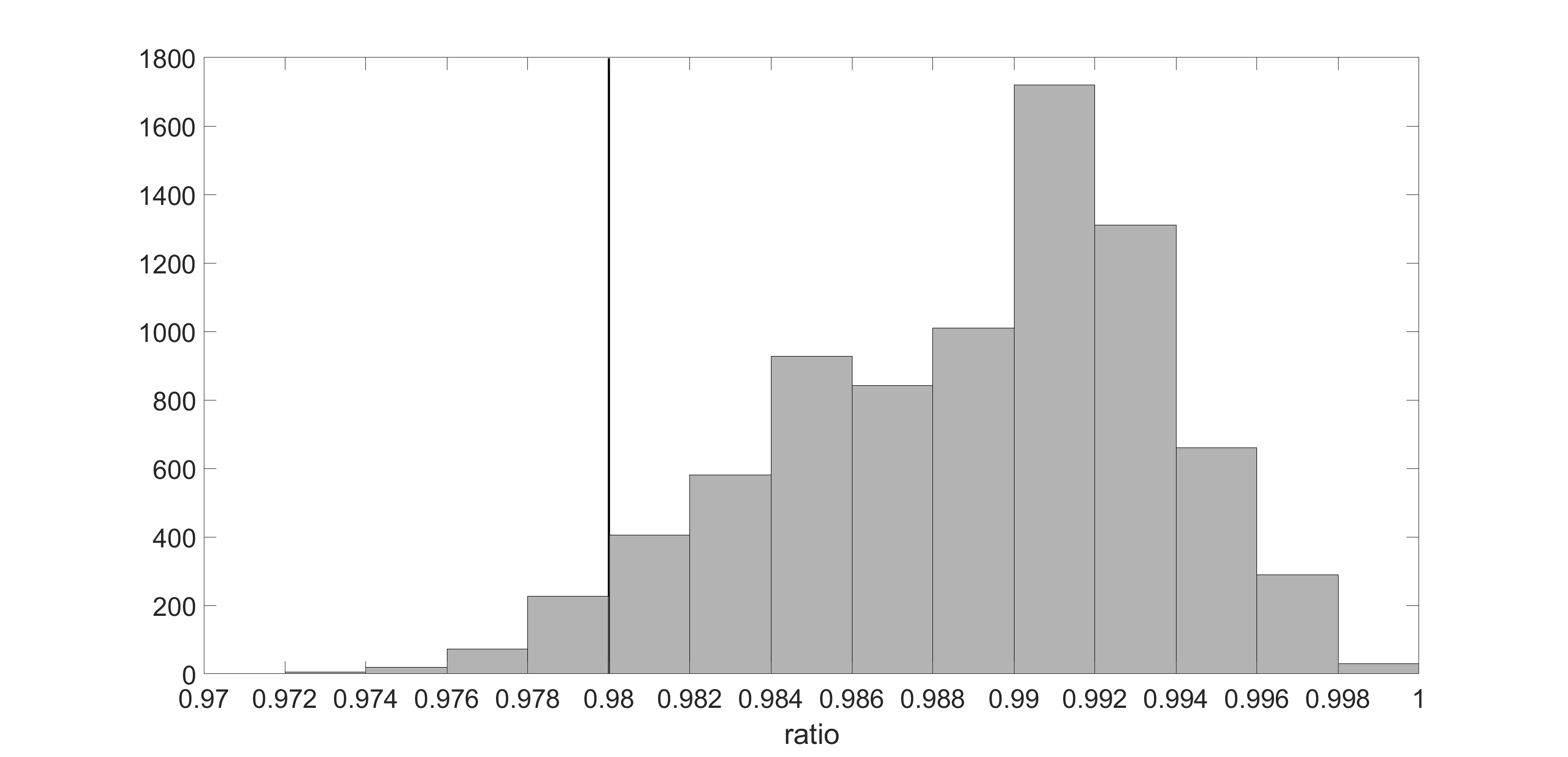}
\end{figure}

\clearpage

\begin{algorithm}
\caption{\small Pseudo-code of the simple iteration with perturbation (SI-P).}
\begin{algorithmic}[1]
\Require initial $\vec x^0$, $L$, $T$
\Ensure $count$, $\vec x^{count}$, $r^{count}$
\State  $r^0\gets 0$, $count\gets 0$
\While{True}\label{alg:whiles}
\For{$ipert=1$ to $L$}
\State $\beta$: choosen from $(0,1)$ randomly \label{beta_random}
\State $(\tilde{\vec x}^{ipert},\tilde{r}^{ipert})\gets \Call{si\_perturb}{\vec x^{count},T,t,\beta}$
\EndFor
\State $ipert\gets\argmax_{ipert\in \{1,2,\ldots,L\}}{\tilde{r}^{ipert}}$\label{alg:loops}
\State $count\gets count+1$\label{code:count}
\If{$\tilde{r}^{ipert}>r^{count-1}$} \label{break0}
\State $\vec x^{count}\gets \tilde{\vec x}^{ipert}$
\State $r^{count}\gets \tilde{r}^{ipert}$
\Else 
\State break 
\EndIf \label{alg:loope}

\EndWhile\label{alg:whilee}
\Function {($\vec x^{opt},r^{opt}$)=si\_perturb}{$\vec x^0,T,t,\beta$}\label{alg:ips}
\State initial $r^0$, $\vec q^0$, $\vec p^0$, $\sigma^0$, $\vec s^0$
\State $k \gets 0$
\State $r^{opt}\gets 0$
\State $\bar{\vec p}^k\gets \Call{p\_bar}{\vec x^k,\vec q^k,\vec p^k}$
\While{$k<T$}
\State $\vec x^{k+1}\gets \Call{exact\_solution}{\vec s^k,r^k}$
\State $r^{k+1} \gets r^k+\delta_F(\vec x^{k+1},\vec x^k)$
\If{$r^{k+1}>r^{opt}$}
\State $r^{opt}\gets r^{k+1}$
\State $\vec x^{opt}\gets \vec x^{k+1}$
\EndIf
\If{$r^{k+1}=r^k=\cdots =r^{k-t}$} \label{trigger0}
\State $\vec x^{k+1}\gets \Call{perturb}{\bar{\vec p}^k,\beta}$\label{alg:pertur}
\State $r^{k+1}\gets r^k+\delta_F(\vec x^{k+1},\vec x^k)$
\EndIf \label{trigger1}
\State $\vec q^{k+1} \gets \vec q^k +\delta_{\vec q}(\vec x^{k+1},\vec x^k)$
\State $\vec p^{k+1} \gets \vec p^k + \delta_{\vec p}(\vec x^{k+1},\vec x^k)$
\State $\bar{\vec p}^{k+1}\gets \Call{p\_bar}{\vec x^{k+1},\vec q^{k+1},\vec p^{k+1}}$
\State $(\vec s^{k+1},\sigma^{k+1}) \gets \Call{subgradient}{\vec x^{k+1},\bar{\vec p}^{k+1},\vec s^{k},\sigma^{k}}$
\State $k\gets k+1$
\EndWhile
\EndFunction\label{alg:ipe}
\end{algorithmic}
\label{alg:ipwp}
\end{algorithm}

\clearpage

\begin{table}
  \centering
\caption{\small  Improved cut values achieved by the simple iteration with perturbation (SI-P) depicted in Alg.~\ref{alg:ipwp}
with $t=3$, $L=20$, $T=2000$, and $p=\infty$. The ratios between the best cut values by SI-P and {the reference ones 
by combinatorial algorithms} \cite{MaHao2017} are at least  $0.997$ (see G37).  
The number of turns of outer loop, recorded by $count$ in Line~\ref{code:count} of Alg.~\ref{alg:ipwp},  is at most  $17$, thereby meaning that the SI-P algorithm runs 
at most  $count\times L\times T=17\times 20 \times 2000 = 680000$ iterations.
}
      \small
\setlength{\tabcolsep}{2.5mm}{
\begin{tabular}{|l|c|c|c|c|c|}
\hline
    \multirow{2}*{\makecell{graph}}  &\multirow{2}*{\makecell{{reference}}}  & \multirow{2}*{\makecell{result}}  &\multirow{2}*{\makecell{ratio}} &\multirow{2}*{\makecell{ratio without \\perturbation}}&\multirow{2}*{\makecell{$count$}} \\
    &&&&&\\
    \hline
    G1    & 11624 & 11624 & 1.0000  & 0.9939  & 9 \\
    G2    & 11620 & 11620 & 1.0000  & 0.9946  & 4 \\
    G3    & 11622 & 11622 & 1.0000  & 0.9985  & 3 \\
    G4    & 11646 & 11646 & 1.0000  & 0.9959  & 17 \\
    G5    & 11631 & 11630 & 0.9999  & 0.9975  & 9 \\
    G14   & 3064  & 3063  & 0.9997  & 0.9899  & 5 \\
    G15   & 3050  & 3050  & 1.0000  & 0.9862  & 12 \\
    G16   & 3052  & 3052  & 1.0000  & 0.9885  & 6 \\
    G17   & 3047  & 3046  & 0.9997  & 0.9879  & 14 \\
    G22   & 13359 & 13358 & 0.9999  & 0.9948  & 8 \\
    G23   & 13344 & 13339 & 0.9996  & 0.9940  & 7 \\
    G24   & 13337 & 13335 & 0.9999  & 0.9963  & 7 \\
    G25   & 13340 & 13337 & 0.9998  & 0.9924  & 4 \\
    G26   & 13328 & 13318 & 0.9992  & 0.9911  & 3 \\
    G35   & 7687  & 7663  & 0.9969  & 0.9869  & 11 \\
    G36   & 7680  & 7656  & 0.9969  & 0.9865  & 13 \\
    G37   & 7691  & 7665  & 0.9966  & 0.9884  & 7 \\
    G38   & 7688  & 7673  & 0.9980  & 0.9876  & 12 \\
    G43   & 6660  & 6660  & 1.0000  & 0.9976  & 2 \\
    G44   & 6650  & 6650  & 1.0000  & 0.9956  & 4 \\
    G45   & 6654  & 6654  & 1.0000  & 0.9917  & 2 \\
    G46   & 6649  & 6646  & 0.9995  & 0.9929  & 5 \\
    G47   & 6657  & 6657  & 1.0000  & 0.9929  & 3 \\
    G51   & 3848  & 3841  & 0.9982  & 0.9891  & 4 \\
    G52   & 3851  & 3849  & 0.9995  & 0.9920  & 8 \\
    G53   & 3850  & 3846  & 0.9990  & 0.9914  & 10 \\
    G54   & 3852  & 3845  & 0.9982  & 0.9920  & 9 \\
    \hline
    \end{tabular}%
    }
  \label{tab:adl}
\end{table}%

\clearpage

\begin{figure}[h]
  \centering
 \caption{\small Comparison with CirCut: The minimum, mean, and maximum cut values (normalized by the number of edges $|E|$) produced by CirCut0, SI,
and CirCut0+SI from multiple starting points are displayed in (a), while those obtained by CirCut, SI$\_$PERTURB,
and CirCut+SI are presented in (b). CirCut0 refers to the pure rank-two relaxation which consists of the line-search and Procedure-CUT, only the first two steps of Alg.~1 in \cite{BurerMonteiroZhang2001}. CirCut0+SI means the output of CirCut0 serves as the input to SI for possible solution quality improvements and so does CirCut+SI.  }
 \label{fig3}
 \subfigure[SI {\it vs} CirCut0.]{\includegraphics[scale=0.35]{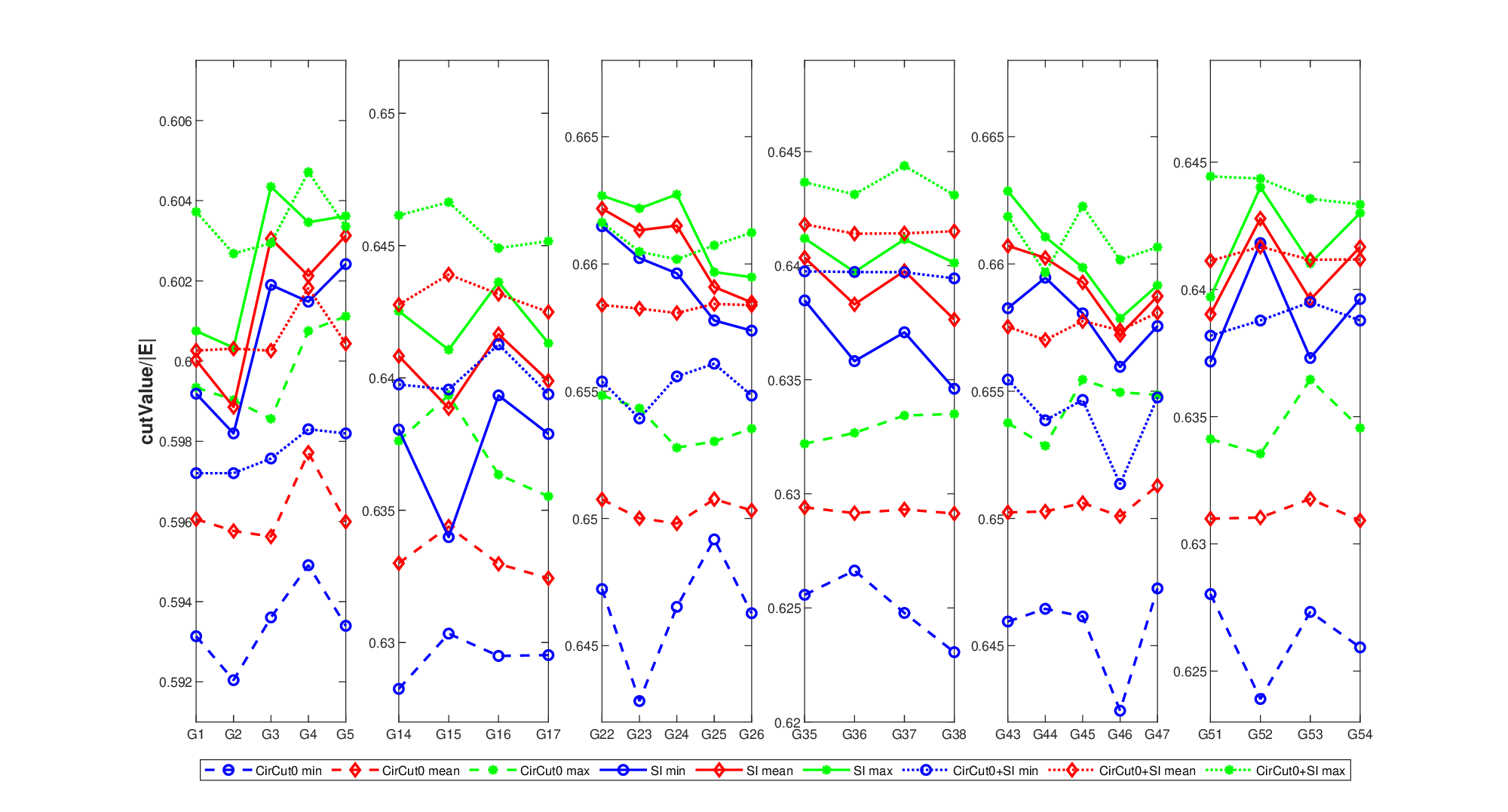}}
 \subfigure[SI$\_$PERTURB {\it vs} CirCut.]{\includegraphics[scale=0.35]{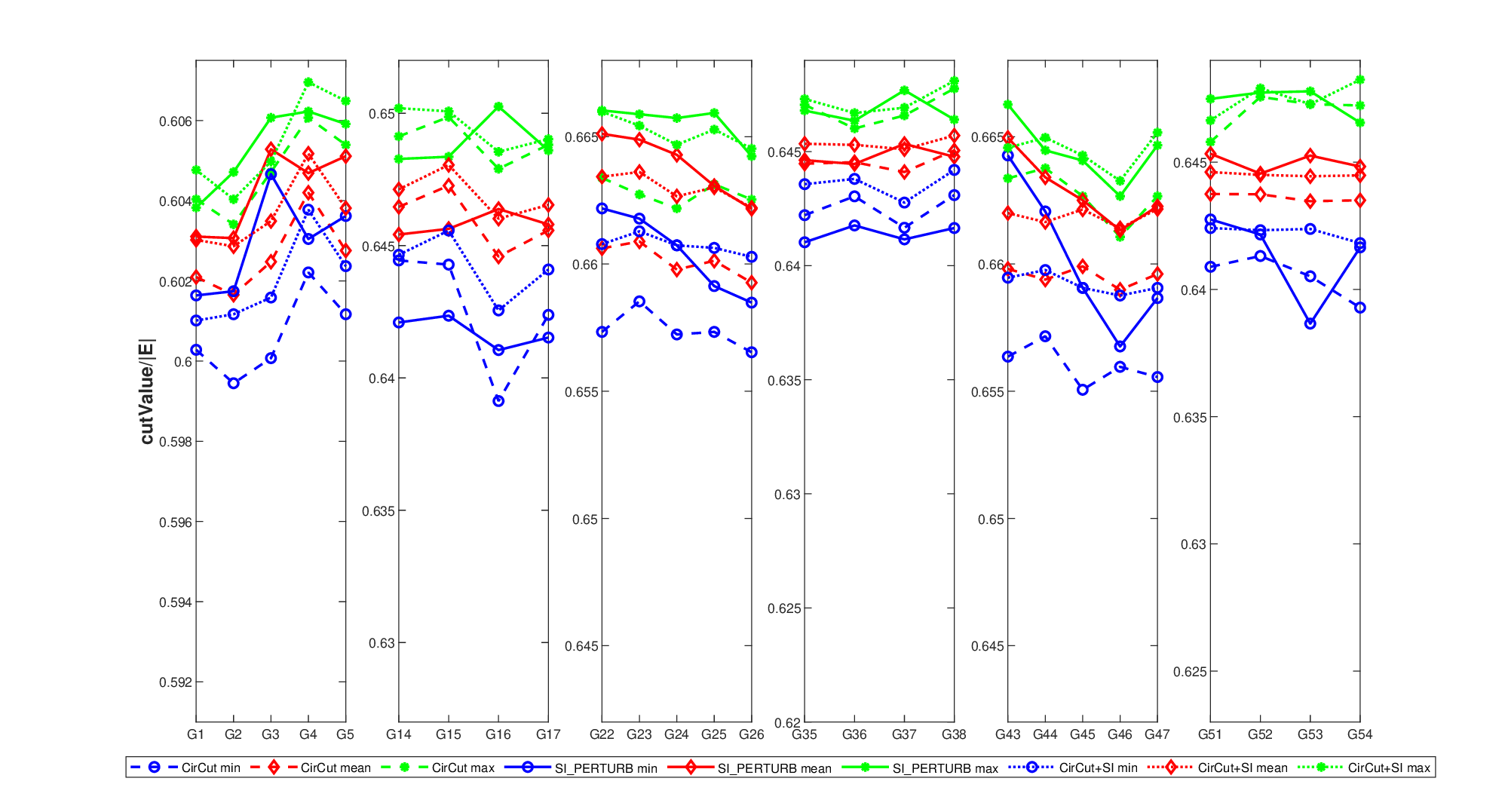}}
\end{figure}

\clearpage

\begin{table}
  \centering
\caption{\small Comparison with CirCut: \blue{The wall-clock time in seconds}, the numbers of calculating $f(\theta)$ and its gradient $\bar{T}_f$ and  $\bar{T}_g$ in CirCut0 and CirCut,  the number of iterations $\bar{T}_S$ in SI and SI$\_$PERTURB,
and the same number of perturbations $\bar{T}_P$ in CirCut and SI$\_$PERTURB. 
It is clearly shown that SI and SI$\_$PERTURB require much fewer iterations and \blue{thus run much faster} than CirCut0 and CirCut, respectively. CirCut0 refers to the pure rank-two relaxation which only contains the first two steps of Alg.~1 in \cite{BurerMonteiroZhang2001}. }
      \small
\setlength{\tabcolsep}{1mm}{
\begin{tabular}{|c|ccc|cc|ccc|ccc|}
\hline
\multirow{3}{*}{graph} & \multicolumn{3}{c|}{CirCut0}                                                  & \multicolumn{2}{c|}{SI}                    & \multicolumn{3}{c|}{CirCut}                                                   & \multicolumn{3}{c|}{SI$\_$PERTURB}                                            \\ \cline{2-12} 
                   & \multicolumn{1}{c|}{\multirow{2}{*}{\blue{time}}   } & \multicolumn{1}{c|}{\multirow{2}{*}{$\bar{T}_g$}} & \multicolumn{1}{c|}{\multirow{2}{*}{$\bar{T}_f$}}  &\multicolumn{1}{c|}{ \multirow{2}{*}{\blue{time}}}& \multicolumn{1}{c|}{\multirow{2}{*}{$\bar{T}_s$}} & \multicolumn{1}{c|}{\multirow{2}{*}{\blue{time}}} & \multicolumn{1}{c|}{\multirow{2}{*}{$\bar{T}_g$}} & \multicolumn{1}{c|}{\multirow{2}{*}{$\bar{T}_f$}} & \multicolumn{1}{c|}{\multirow{2}{*}{\blue{time}}} & \multicolumn{1}{c|}{\multirow{2}{*}{$\bar{T}_s$}}  & \multirow{2}{*}{$\bar{T}_p$} \\ 
                   & \multicolumn{1}{c|}{}                             &        \multicolumn{1}{c|}{}                                     &        \multicolumn{1}{c|}{}                                 &                                                 \multicolumn{1}{c|}{}                             & \multicolumn{1}{c|}{}          &     \multicolumn{1}{c|}{}                                          &  \multicolumn{1}{c|}{}                                                        & \multicolumn{1}{c|}{}            &  \multicolumn{1}{c|}{}                                           &   \multicolumn{1}{c|}{}            &\multicolumn{1}{c|}{}                                                                       \\     
                       \hline
G1                     & \multicolumn{1}{c|}{1.73}                         & \multicolumn{1}{c|}{74.45}                        & 126                      & \multicolumn{1}{c|}{0.34}                         & 32.65                    & \multicolumn{1}{c|}{39.37}                        & \multicolumn{1}{c|}{1589.05}                      & 2817.9                   & \multicolumn{1}{c|}{0.61}                         & \multicolumn{1}{c|}{643.55}                   & 30                           \\
G2                     & \multicolumn{1}{c|}{1.76}                         & \multicolumn{1}{c|}{74.75}                        & 125.55                   & \multicolumn{1}{c|}{0.37}                         & 33.15                    & \multicolumn{1}{c|}{38.33}                        & \multicolumn{1}{c|}{1544}                         & 2763.15                  & \multicolumn{1}{c|}{0.55}                         & \multicolumn{1}{c|}{652.9}                    & 30                           \\
G3                     & \multicolumn{1}{c|}{1.75}                         & \multicolumn{1}{c|}{73.85}                        & 124.35                   & \multicolumn{1}{c|}{0.31}                         & 32.1                     & \multicolumn{1}{c|}{41.82}                        & \multicolumn{1}{c|}{1745.25}                      & 3103.15                  & \multicolumn{1}{c|}{0.50}                         & \multicolumn{1}{c|}{637.25}                   & 33                           \\
G4                     & \multicolumn{1}{c|}{2.09}                         & \multicolumn{1}{c|}{92.95}                        & 156.85                   & \multicolumn{1}{c|}{0.32}                         & 34.65                    & \multicolumn{1}{c|}{37.95}                        & \multicolumn{1}{c|}{1585.1}                       & 2788.6                   & \multicolumn{1}{c|}{0.54}                         & \multicolumn{1}{c|}{617.7}                    & 30                           \\
G5                     & \multicolumn{1}{c|}{1.88}                         & \multicolumn{1}{c|}{77}                           & 130                      & \multicolumn{1}{c|}{0.33}                         & 29                       & \multicolumn{1}{c|}{39.18}                        & \multicolumn{1}{c|}{1657.25}                      & 2909.15                  & \multicolumn{1}{c|}{0.56}                         & \multicolumn{1}{c|}{677.9}                    & 31                           \\
G14                    & \multicolumn{1}{c|}{1.09}                         & \multicolumn{1}{c|}{75}                           & 118.3                    & \multicolumn{1}{c|}{0.10}                         & 24.15                    & \multicolumn{1}{c|}{27.09}                        & \multicolumn{1}{c|}{2070.7}                       & 3183.85                  & \multicolumn{1}{c|}{0.60}                         & \multicolumn{1}{c|}{791.1}                    & 34                           \\
G15                    & \multicolumn{1}{c|}{1.09}                         & \multicolumn{1}{c|}{75.5}                         & 118                      & \multicolumn{1}{c|}{0.11}                         & 28.25                    & \multicolumn{1}{c|}{26.14}                        & \multicolumn{1}{c|}{2038.9}                       & 3137.1                   & \multicolumn{1}{c|}{0.58}                         & \multicolumn{1}{c|}{780.15}                   & 32                           \\
G16                    & \multicolumn{1}{c|}{1.17}                         & \multicolumn{1}{c|}{81.4}                         & 127.5                    & \multicolumn{1}{c|}{0.11}                         & 25                       & \multicolumn{1}{c|}{22.18}                        & \multicolumn{1}{c|}{1695.25}                      & 2628.1                   & \multicolumn{1}{c|}{0.46}                         & \multicolumn{1}{c|}{660.8}                    & 28                           \\
G17                    & \multicolumn{1}{c|}{1.00}                         & \multicolumn{1}{c|}{76.25}                        & 119.35                   & \multicolumn{1}{c|}{0.10}                         & 23.3                     & \multicolumn{1}{c|}{22.75}                        & \multicolumn{1}{c|}{1752}                         & 2694.95                  & \multicolumn{1}{c|}{0.45}                         & \multicolumn{1}{c|}{658.35}                   & 29                           \\
G22                    & \multicolumn{1}{c|}{4.26}                         & \multicolumn{1}{c|}{65.6}                         & 101.9                    & \multicolumn{1}{c|}{0.90}                         & 31.05                    & \multicolumn{1}{c|}{90.58}                        & \multicolumn{1}{c|}{1344.4}                       & 2076.35                  & \multicolumn{1}{c|}{3.81}                         & \multicolumn{1}{c|}{801.15}                   & 31                           \\
G23                    & \multicolumn{1}{c|}{3.51}                         & \multicolumn{1}{c|}{54.4}                         & 84.7                     & \multicolumn{1}{c|}{0.91}                         & 31.25                    & \multicolumn{1}{c|}{104.28}                       & \multicolumn{1}{c|}{1538.15}                      & 2395.05                  & \multicolumn{1}{c|}{3.74}                         & \multicolumn{1}{c|}{795.7}                    & 34                           \\
G24                    & \multicolumn{1}{c|}{3.94}                         & \multicolumn{1}{c|}{62.35}                        & 96.3                     & \multicolumn{1}{c|}{0.91}                         & 32.9                     & \multicolumn{1}{c|}{94.22}                        & \multicolumn{1}{c|}{1404.1}                       & 2175.95                  & \multicolumn{1}{c|}{3.69}                         & \multicolumn{1}{c|}{805.25}                   & 31                           \\
G25                    & \multicolumn{1}{c|}{3.95}                         & \multicolumn{1}{c|}{63}                           & 97.7                     & \multicolumn{1}{c|}{0.90}                         & 32.45                    & \multicolumn{1}{c|}{87.42}                        & \multicolumn{1}{c|}{1297.55}                      & 2007.25                  & \multicolumn{1}{c|}{3.50}                         & \multicolumn{1}{c|}{732.9}                    & 29                           \\
G26                    & \multicolumn{1}{c|}{4.02}                         & \multicolumn{1}{c|}{63.8}                         & 99.9                     & \multicolumn{1}{c|}{0.86}                         & 32.45                    & \multicolumn{1}{c|}{83.02}                        & \multicolumn{1}{c|}{1244.25}                      & 1916.05                  & \multicolumn{1}{c|}{3.64}                         & \multicolumn{1}{c|}{750.6}                    & 28                           \\
G35                    & \multicolumn{1}{c|}{4.50}                         & \multicolumn{1}{c|}{86.15}                        & 136.7                    & \multicolumn{1}{c|}{0.58}                         & 27.6                     & \multicolumn{1}{c|}{155.32}                       & \multicolumn{1}{c|}{2968.25}                      & 4606.1                   & \multicolumn{1}{c|}{5.57}                         & \multicolumn{1}{c|}{1020}                     & 40                           \\
G36                    & \multicolumn{1}{c|}{4.28}                         & \multicolumn{1}{c|}{83.6}                         & 131.2                    & \multicolumn{1}{c|}{0.64}                         & 30.45                    & \multicolumn{1}{c|}{176.06}                       & \multicolumn{1}{c|}{3405.65}                      & 5276.25                  & \multicolumn{1}{c|}{6.67}                         & \multicolumn{1}{c|}{1072.1}                   & 42                           \\
G37                    & \multicolumn{1}{c|}{4.67}                         & \multicolumn{1}{c|}{88.15}                        & 141.5                    & \multicolumn{1}{c|}{0.64}                         & 31.1                     & \multicolumn{1}{c|}{165.54}                       & \multicolumn{1}{c|}{3206.85}                      & 4959.4                   & \multicolumn{1}{c|}{6.44}                         & \multicolumn{1}{c|}{1042.5}                   & 41                           \\
G38                    & \multicolumn{1}{c|}{4.87}                         & \multicolumn{1}{c|}{92.75}                        & 145.7                    & \multicolumn{1}{c|}{0.64}                         & 29.9                     & \multicolumn{1}{c|}{164.64}                       & \multicolumn{1}{c|}{3217.9}                       & 4957.6                   & \multicolumn{1}{c|}{7.35}                         & \multicolumn{1}{c|}{1170.75}                  & 43                           \\
G43                    & \multicolumn{1}{c|}{1.39}                         & \multicolumn{1}{c|}{62.55}                        & 97.45                    & \multicolumn{1}{c|}{0.24}                         & 33.35                    & \multicolumn{1}{c|}{29.69}                        & \multicolumn{1}{c|}{1275.85}                      & 1998.3                   & \multicolumn{1}{c|}{0.62}                         & \multicolumn{1}{c|}{598.75}                   & 29                           \\
G44                    & \multicolumn{1}{c|}{1.41}                         & \multicolumn{1}{c|}{63.25}                        & 99.1                     & \multicolumn{1}{c|}{0.24}                         & 29.9                     & \multicolumn{1}{c|}{34.89}                        & \multicolumn{1}{c|}{1517.6}                       & 2402.25                  & \multicolumn{1}{c|}{0.72}                         & \multicolumn{1}{c|}{746.2}                    & 34                           \\
G45                    & \multicolumn{1}{c|}{1.47}                         & \multicolumn{1}{c|}{65.55}                        & 102.6                    & \multicolumn{1}{c|}{0.24}                         & 28.2                     & \multicolumn{1}{c|}{33.33}                        & \multicolumn{1}{c|}{1458.75}                      & 2319.2                   & \multicolumn{1}{c|}{0.69}                         & \multicolumn{1}{c|}{739.2}                    & 32                           \\
G46                    & \multicolumn{1}{c|}{1.37}                         & \multicolumn{1}{c|}{60.95}                        & 95.15                    & \multicolumn{1}{c|}{0.24}                         & 27.35                    & \multicolumn{1}{c|}{29.63}                        & \multicolumn{1}{c|}{1262.55}                      & 2002.3                   & \multicolumn{1}{c|}{0.55}                         & \multicolumn{1}{c|}{606.75}                   & 28                           \\
G47                    & \multicolumn{1}{c|}{1.37}                         & \multicolumn{1}{c|}{64.25}                        & 99.95                    & \multicolumn{1}{c|}{0.23}                         & 27.55                    & \multicolumn{1}{c|}{25.59}                        & \multicolumn{1}{c|}{1127.9}                       & 1777.85                  & \multicolumn{1}{c|}{0.58}                         & \multicolumn{1}{c|}{624.6}                    & 25                           \\
G51                    & \multicolumn{1}{c|}{1.56}                         & \multicolumn{1}{c|}{82.5}                         & 129.95                   & \multicolumn{1}{c|}{0.15}                         & 25.1                     & \multicolumn{1}{c|}{41.71}                        & \multicolumn{1}{c|}{2281.45}                      & 3516.85                  & \multicolumn{1}{c|}{1.02}                         & \multicolumn{1}{c|}{793.35}                   & 35                           \\
G52                    & \multicolumn{1}{c|}{1.50}                         & \multicolumn{1}{c|}{79.85}                        & 125.55                   & \multicolumn{1}{c|}{0.15}                         & 24.3                     & \multicolumn{1}{c|}{45.90}                        & \multicolumn{1}{c|}{2506.95}                      & 3875                     & \multicolumn{1}{c|}{1.22}                         & \multicolumn{1}{c|}{922.45}                   & 38                           \\
G53                    & \multicolumn{1}{c|}{1.66}                         & \multicolumn{1}{c|}{88.85}                        & 139.55                   & \multicolumn{1}{c|}{0.17}                         & 25.4                     & \multicolumn{1}{c|}{36.32}                        & \multicolumn{1}{c|}{1959.1}                       & 3017.85                  & \multicolumn{1}{c|}{0.93}                         & \multicolumn{1}{c|}{753.15}                   & 32                           \\
G54                    & \multicolumn{1}{c|}{1.49}                         & \multicolumn{1}{c|}{79.65}                        & 124.25                   & \multicolumn{1}{c|}{0.15}                         & 23.25                    & \multicolumn{1}{c|}{35.04}                        & \multicolumn{1}{c|}{1873.75}                      & 2885.05                  & \multicolumn{1}{c|}{0.88}                         & \multicolumn{1}{c|}{690.9}                    & 30                           \\   \hline
\end{tabular}
 }
  \label{tab:circutcost}\end{table}%

\clearpage


\begin{thebibliography}{BGJR88}

\bibitem[AS16]{AizenbudShkolnisky2016}
Y.~Aizenbud and Y.~Shkolnisky.
\newblock A max-cut approach to heterogeneity in cryo-electron microscopy.
\newblock {\em arXiv:1609.00100}, 2016.

\bibitem[BGJR88]{BarahonaGrotschelJungerReinelt1988}
F.~Barahona, M.~Gr{\"o}tschel, M.~J{\"u}nger, and G.~Reinelt.
\newblock An application of combinatorial optimization to statistical physics
  and circuit layout design.
\newblock {\em Oper. Res.}, 36(3):493--513, 1988.

\bibitem[BMZ01]{BurerMonteiroZhang2001}
S.~Burer, R.~D.~C. Monteiro, and Y.~Zhang.
\newblock Rank-two relaxation heuristics for {MAX-CUT} and other binary
  quadratic programs.
\newblock {\em SIAM J. Optim.}, 12:503--521, 2001.

\bibitem[CD87]{ChangDu1987}
K.~C. Chang and D.~H.~C. Du.
\newblock Efficient algorithms for layer assignment problem.
\newblock {\em IEEE Trans. Comput. Aided Design}, 6(1):67--78, 1987.

\bibitem[CSZ16]{ChangShaoZhang2016-maxcut}
K.~C. Chang, S.~Shao, and D.~Zhang.
\newblock {Spectrum of the signless $1$-Laplacian and the dual {Cheeger}
  constant on graphs}.
\newblock {\em arXiv:1607.00489}, 2016.

\bibitem[CSZ17]{ChangShaoZhang2017}
K.~C. Chang, S.~Shao, and D.~Zhang.
\newblock {Nodal domains of eigenvectors for {1-Laplacian} on graphs}.
\newblock {\em Adv. Math.}, 308:529--574, 2017.

\bibitem[Din67]{Dinkelbach1967}
W.~Dinkelbach.
\newblock On nonlinear fractional programming.
\newblock {\em Manage. Sci.}, 13(7):492--498, 1967.

\bibitem[DP93]{DelormePoljak1993}
C.~Delorme and S.~Poljak.
\newblock Laplacian eigenvalues and the maximum cut problem.
\newblock {\em Math. Program.}, 62:557--574, 1993.

\bibitem[GW95]{GoemansWilliamson1995}
M.~X. Goemans and D.~P. Williamson.
\newblock Improved approximation algorithms for maximum cut and satisfiability
  problems using semidefinite programming.
\newblock {\em J. Assoc. Comput. Mach.}, 42:1115--1145, 1995.

\bibitem[Kar72]{Karp1972}
R.~M. Karp.
\newblock Reducibility among combinatorial problems.
\newblock In R.~E. Miller, J.~W. Thatcher, and J.~D. Bohlinger, editors, {\em
  {Complexity of Computer Computations}}, pages 85--103. Springer, 1972.

\bibitem[LG16]{LinGuan2016}
G.~Lin and J.~Guan.
\newblock An integrated method based on {PSO} and {EDA} for the max-cut
  problem.
\newblock {\em Comput. Intel. Neurosc.}, 2016:1--13, 2016.

\bibitem[MDL09]{MartiDuarteLaguna2009}
R.~Mart{\'i}, A.~Duarte, and M.~Laguna.
\newblock Advanced scatter search for the {Max-Cut} problem.
\newblock {\em INFORMS J. Comput.}, 21:26--38, 2009.

\bibitem[MH17]{MaHao2017}
F.~Ma and J.~Hao.
\newblock A multiple search operator heuristic for the {max-k-cut} problem.
\newblock {\em Ann. Oper. Res.}, 248:365--403, 2017.

\bibitem[Ott08]{th:Ottaviano2008}
G.~Ottaviano.
\newblock {Spectral Approximation Algorithms for Graph Cut Problems}.
\newblock Master's thesis, Universit{\`a} di Pisa, Lungarno Pacinotti, 2008.

\bibitem[PK04]{PalubeckisKrivickiene2004}
G.~Palubeckis and V.~Krivickiene.
\newblock Application of multistart tabu search to the max-cut problem.
\newblock {\em Inf. Technol. Control}, 31:29--35, 2004.

\bibitem[PR95]{PoljakRendl1995}
S.~Poljak and F.~Rendl.
\newblock Solving the max-cut problem using eigenvalues.
\newblock {\em Discrete Appl. Math.}, 62:249--278, 1995.

\bibitem[SI83]{SchaibleIbaraki1983}
S.~Schaible and T.~Ibaraki.
\newblock Fractional programming.
\newblock {\em Eur. J. Oper. Res.}, 12(4):325--338, 1983.

\bibitem[Tre12]{Trevisan2012}
L.~Trevisan.
\newblock Max cut and the smallest eigenvalue.
\newblock {\em SIAM J. Comput.}, 41:1769--1786, 2012.

\end{thebibliography}

          \end{document}